\documentclass[a4paper]{article}
\usepackage{fullpage}
\usepackage{tikz}

\usepackage[pdftex,breaklinks,colorlinks,
linkcolor=blue,
citecolor=blue,
urlcolor=blue]{hyperref}

\usepackage{enumitem}
\usepackage{amsmath,amssymb,amsfonts,amsthm}

\newtheorem{theorem}{Theorem}

\newtheorem{lemma}[theorem]{Lemma}
\newtheorem{proposition}[theorem]{Proposition}
\newtheorem{corollary}[theorem]{Corollary}

\theoremstyle{definition}
\newtheorem{definition}[theorem]{Definition}
\newtheorem{remark}[theorem]{Remark}

\author{Pablo Romero\footnote{Facultad de Ingenier\'ia, Universidad de la Rep\'ublica, Montevideo, Uruguay. E-mail address: \texttt{promero@fing.edu.uy}.}\footnote{Facultad de Ciencias Exactas y Naturales, Universidad de Buenos Aires. E-mail address: \texttt{promero@ic.fcen.uba.ar}.}}

\date{}

\makeatletter%
\begin{document}

\title{Existence of uniformly most reliable two-terminal graphs with distance constraints\thanks{An extended abstract of this work (without proofs) was published in the proceedings of the XII Latin-American Algorithms, Graphs and Optimization Symposium~\cite{RomeroLAGOS2-2025}. Several new results on the existence and nonexistence of $d$-UMRTTGs are included in this article.}}

\maketitle

\begin{abstract}\let\thefootnote\relax
A two-terminal graph is a simple graph equipped with two distinguished vertices, called terminals. Let $T_{n,m}$ be the class consisting of all non-isomorphic two-terminal graphs on $n$ vertices and $m$ edges. Let $G$ be any two-terminal graph in $T_{n,m}$, and let $d$ be any positive integer. For every $\rho\in [0,1]$, the \emph{$d$-constrained two-terminal reliability of $G$ at $\rho$}, denoted $R_G^d(\rho)$, is the probability that $G$ has some path of length at most $d$ joining its terminals after each of its edges is independently deleted with probability $\rho$. We say $G$ is a \emph{$d$-uniformly most reliable two-terminal graph} ($d$-UMRTTG) if for each $H$ in $T_{n,m}$ and every $\rho \in [0,1]$ it holds that $R_{G}^d(\rho)\geq R_{H}^{d}(\rho)$. Previous works studied the existence of $d$-UMRTTG in $T_{n,m}$ when $d$ is greater than or equal to $n-1$, or equivalently, when the distance constraint is dropped. In this work, a 
characterization of all $1$-UMRTTGs and $2$-UMRTTGs is given. Several structural properties of 3-UMRTTGs are established. For every $d\geq 3$, we prove the existence of a unique $d$-UMRTTG in $T_{n,\binom{n}{2}-1}$. 
In contrast, for every $d\geq 4$ and every $n\geq 11$, no $d$-UMRTTG exists in $T_{n,m}$ whenever $20\leq m\leq 3n-9$ or
$3n-5\leq m\leq \binom{n}{2}-2$. Additionally, for each $d\geq 5$, it is proved that no $d$-UMRTTG exists in $T_{n,m}$ when $n\geq 11$ and $m=3n-8$. 
\end{abstract}

\renewcommand{\labelitemi}{--}

\section{Introduction} \label{intro}
Let $T_{n,m}$ be the set of all non-isomorphic two-terminal graphs on $n$ vertices and $m$ edges. Let $G\in T_{n,m}$ and let $d$ be a positive integer. A spanning subgraph of $G$ is called a \emph{$d$-pathset} if it contains a path of length at most $d$ joining its terminals. For each $\rho\in[0,1]$, the \emph{$d$-constrained two-terminal reliability of $G$ at $\rho$}, denoted by $R_G^d(\rho)$, is the probability that the resulting spanning subgraph is a $d$-pathset after each edge of $G$ is independently deleted with probability $\rho$.

Given positive integers $d$, $n$, and $m$ such that $n\geq 2$ and $1\leq m\leq \binom{n}{2}$, we investigate whether there exists a two-terminal graph $G\in T_{n,m}$ satisfying $R_G^d(\rho) \geq R_H^d(\rho)$ for every $H$ in $T_{n,m}$ and every $\rho$ in $[0,1]$. Such a two-terminal graph $G$ is called a \emph{$d$-uniformly most reliable two-terminal graph} ($d$-UMRTTG).  

In this work, we study the existence of $d$-UMRTTGs in each nonempty class $T_{n,m}$ as a function of the distance constraint $d$. This article is organized as follows. Section~\ref{section:concepts} introduces the basic concepts and notation, while   Section~\ref{section:background} reviews related work. Section~\ref{section:1and2} characterizes all $1$-UMRTTGs and $2$-UMRTTGs in each nonempty set $T_{n,m}$. Section~\ref{section:3} establishes structural properties satisfied by all $3$-UMRTTGs and proves that there is a unique $d$-UMRTTG in $T_{n,\binom{n}{2}-1}$ when $d\geq 3$.

Whereas $d$-UMRTTGs always exist when 
$d\in \{1,2\}$, Section~\ref{section:contrast} shows that no $d$-UMRTTG exists in $T_{n,m}$ whenever $d\geq 4$, $n\geq 11$, and either $20\leq m\leq 3n-9$ or $3n-5\leq m\leq \binom{n}{2}-2$. Additionally, it is shown that no $d$-UMRTTG exists in $T_{n,3n-8}$ when $n\geq 11$ and $d\geq 5$. Finally, Section~\ref{section:conclusions} presents concluding remarks and open problems.

\section{Concepts}\label{section:concepts}
All graphs in this work are finite, simple, and undirected. Let $G$ be a graph. We denote the vertex set and edge set of $G$ by $V(G)$ and $E(G)$, respectively. Let $u$ and $v$ be in $V(G)$. We say that $u$ and $v$ are \emph{adjacent} if $uv\in E(G)$. The \emph{open neighborhood of $v$}, denoted by $N_G(v)$, is the set of all vertices adjacent to $v$. The \emph{degree of $v$ in $G$}, denoted by $d_G(v)$, is defined as $|N_G(v)|$.

A vertex $v$ is \emph{isolated} if $d_G(v)=0$, and \emph{universal} if it is adjacent to every other vertex in $V(G)$. The graph $G$ is \emph{regular} if all its vertices have the same degree, and \emph{almost-regular} if the degrees of any two vertices differ by at most one.

For each subset $W\subseteq V(G)$ and each vertex $u\in V(G)$, we denote by $uW$ the set of all edges $uw\in E(G)$ such that $w\in W$. We denote the path on $n$ vertices and the complete graph on $n$ vertices by $P_n$ and $K_n$, respectively. The \emph{length} of $P_n$ is $n-1$.

The \emph{distance} between $u$ and $v$ in $G$, denoted by $d_G(u,v)$, is the length of a shortest path in $G$ joining $u$ and $v$ (or infinity if no such path exists). A \emph{subgraph} of $G$ is a graph $H$ satisfying $V(H)\subseteq V(G)$ and $E(H)\subseteq E(G)$. If $G\neq K_1$, the \emph{edge-connectivity} of $G$ is the minimum number of edges whose removal disconnects $G$.

Two graphs $G$ and $H$ are \emph{isomorphic} if there exists a bijection $\varphi:V(G)\to V(H)$ such that $uv\in E(G)$ if and only if $\varphi(u)\varphi(v)\in E(H)$. Such a mapping $\varphi$ is called an \emph{isomorphism}. For each positive integer $n$, we denote by $p_n(G)$ the number of subgraphs of $G$ isomorphic to $P_n$.

A \emph{two-terminal graph} is a graph $G$ equipped with two distinguished vertices called \emph{terminals}. Throughout this work, the symbols $s$ and $t$ denote the terminal vertices of a two-terminal graph. A pair of two-terminal graphs $G$ and $H$ are \emph{isomorphic} if there exists an isomorphism $\varphi$ between $G$ and $H$ that preserves the set of terminals; that is, either $\varphi(s)=s$ and $\varphi(t)=t$, or $\varphi(s)=t$ and $\varphi(t)=s$.

\section{Related work}\label{section:background}
Let $G$ be any two-terminal graph in $T_{n,m}$ and let $d$ be any positive integer. A spanning subgraph $H$ of $G$ is a \emph{$d$-pathset} if $H$ has some path of length at most $d$ joining its terminals. For each $i\in\{1,2,\ldots,m\}$, let $N_i^d(G)$ be the number of $d$-pathsets of $G$ having exactly $i$ edges. For every $\rho$ in $[0,1]$, the $d$-constrained two-terminal reliability of $G$ can be written as follows:
\begin{equation}\label{equation:poly}
R_G^d(\rho) = \sum_{i=1}^{m}N_i^d(G)(1-\rho)^i\rho^{m-i}.    
\end{equation}

For convenience, we will sometimes consider the \emph{$d$-constrained two-terminal unreliability of $G$}, denoted $U_G^d(\rho)$, which is defined as $1-R_G^d(\rho)$. The key concept of this work is the following. 

\begin{definition}
Let $d$ be any positive integer. The two-terminal graph $G$ in $T_{n,m}$ is a \emph{$d$-uniformly most reliable two-terminal graph} ($d$-UMRTTG) if for each $H$ in $T_{n,m}$ and for every $\rho$ in $[0,1]$ it holds that $R_G^d(\rho)\geq R_H^d(\rho)$.     
\end{definition}

Given positive integers $d$, $n$ and $m$ such that $n\geq 2$ and $1\leq m\leq \binom{n}{2}$, we investigate whether there exists a $d$-UMRTTG in $T_{n,m}$. The following concept gives a sufficient condition for a two-terminal graph $G$ in $T_{n,m}$ to be $d$-UMRTTG.

\begin{definition}
Let $G$ and $H$ be in $T_{n,m}$ and let $d$ be a positive integer. We say that $G$ is \emph{$d$-stronger than $H$} if $N_i^d(G)\geq N_i^d(H)$ for every $i\in\{1,2,\ldots,m\}$ and $N_j^d(G)>N_j^d(H)$ for some $j\in\{1,2,\ldots,m\}$.
\end{definition}
A typical way to prove that $G$ is $d$-UMRTTG is to prove that $G$ is $d$-stronger than any other two-terminal graph $H$ in $T_{n,m}$. The proof of Lemma~\ref{lemma:dstrongimpliesdUMRTTG} is elementary; however, we include it for completeness. 
\begin{lemma}\label{lemma:dstrongimpliesdUMRTTG}
If $G$ is $d$-stronger than every $H$ in $T_{n,m}-\{G\}$ then $G$ is $d$-UMRTTG in $T_{n,m}$.
\end{lemma}
\begin{proof}
Let $d$ be any positive integer, and let $H$ be in $T_{n,m}-\{G\}$. 
As $G$ is $d$-stronger than $H$, it follows that $N_i^d(G)\geq N_i^d(H)$ and further, $N_j^d(G)>N_j^d(H)$ for some $j\in \{1,2,\ldots,m\}$. Equation~\eqref{equation:poly} gives that, for every $\rho \in (0,1)$, 
\begin{align*}
R_G^d(\rho) &= N_j^d(G)(1-\rho)^j\rho^{m-j}+
\sum_{i\in \{1,2,\ldots,m\}-\{j\}}N_i^d(G)(1-\rho)^i\rho^{m-i}\\
&>N_j^d(H)(1-\rho)^j\rho^{m-j} + \sum_{i\in \{1,2,\ldots,m\}-\{j\}}N_i^d(G)(1-\rho)^i\rho^{m-i}\\
&\geq N_j^d(H)(1-\rho)^j\rho^{m-j}+\sum_{i\in \{1,2,\ldots,m\}-\{j\}}N_i^d(H)(1-\rho)^i\rho^{m-i}=R_H^d(\rho).
\end{align*}
It suffices to prove that if $\rho$ is either $0$ or $1$ 
then $R_G^d(\rho)\geq R_H^d(\rho)$. 
If $\rho=1$ then all edges are deleted and there is no path joining the terminals hence $R_G^d(1)=R_H^d(1)=0$. As $N_j^d(G)>N_j^d(H)$ for some $j\in \{1,2,\ldots,m\}$, 
there is $d$-pathset of $G$ with $j$ edges. 
Therefore, if $\rho=0$ then no edges are deleted and $R_G^d(\rho)=1\geq R_H^d(\rho)$. 
We proved that for each $H$ in $T_{n,m}-\{G\}$ and for every $\rho$ in $[0,1]$ it holds that 
$R_G^d(\rho)\geq R_H^d(\rho)$. Hence, $G$ is $d$-UMRTTG, and the lemma follows.
\end{proof}

The following concepts give necessary conditions for a two-terminal graph $G$ in $T_{n,m}$ to be a $d$-UMRTTG.

\begin{definition}
Let $G$ be a two-terminal graph in $T_{n,m}$ and let $d$ be any positive integer. 
\begin{enumerate}[label=(\roman*)]
\item We say $G$ is a \emph{$d$-locally most reliable two-terminal graph near $0$} ($d$-LMRTTG near $0$) if for each $H$ in $T_{n,m}$ there exists $\delta>0$ such that  for every $\rho$ in $(0,\delta)$ it holds that $R_G^d(\rho)\geq R_H^d(\rho)$.     
\item We say $G$ is a \emph{$d$-locally most reliable two-terminal graph near $1$} ($d$-LMRTTG near $1$) if for each $H$ in $T_{n,m}$ there exists $\delta>0$ such that for every $\rho$ in $(1-\delta,1)$ it holds that $R_G^d(\rho)\geq R_H^d(\rho)$. 
\end{enumerate}
\end{definition}

Each $d$-UMRTTG is simultaneously $d$-LMRTTG near $0$ and near $1$. Remark~\ref{remark:local} is an essential tool to find $d$-LMRTTGs near $0$ or $1$. Its proof follows from elementary calculus and equation~\eqref{equation:poly}.
\begin{remark}[Brown and Cox~\cite{Brown-2014}]\label{remark:local}
Let $G$ and $H$ be in $T_{n,m}$ and 
let 
$d$ be a positive integer.
\begin{enumerate}[label=(\roman*)]
    \item\label{it1:local} If $N_i^d(G)=N_i^d(H)$ for all $i\in \{1,2,\ldots,k-1\}$ and  
    $N_k^d(G)>N_k^d(H)$, then there exists $\delta>0$ such that for every $\rho\in (1-\delta,1)$ it holds that $R_G^d(\rho)> R_H^d(\rho)$.
    \item\label{it2:local} If $N_{i}^d(G)=N_{i}^d(H)$ for all $i\in \{j+1,j+2,\ldots,m\}$ and  
    $N_{j}^d(G)>N_{j}^d(H)$, then there exists $\delta>0$ such that for every $\rho\in (0,\delta)$ it holds that $R_G^d(\rho)> R_H^d(\rho)$.
\end{enumerate}
\end{remark}

Observe that, for each nonempty set $T_{n,m}$ and each positive integer $d$, by finiteness there always exists $G$ in $T_{n,m}$ that is a $d$-LMRTTG near $0$ as well as a $d$-LMRTTG near $1$. Let us illustrate how to construct the set consisting of all $d$-LMRTTGs near $1$ (the reasoning to construct all $d$-LMRTTGs near $0$ is analogous). 
Define $T_{n,m}^{d}(1)=T_{n,m}$ and, for each $i\in \{1,2,\ldots,m-1\}$ define $T_{n,m}^d(i+1)$ as follows,
\begin{equation*}
T_{n,m}^d(i+1)=\{G: G\in T_{n,m}^d(i), \, N_{i+1}^d(G)\geq N_{i+1}^d(H) \text{ for each } H \text{ in } T_{n,m}^d(i)\}.
\end{equation*}
That is, at stage $i+1$ we keep  
precisely those two-terminal graphs maximizing the number of $d$-pathsets on $i+1$ edges among the two-terminal graphs at stage $i$.
\begin{lemma}\label{lemma:existenceLMRTTG}
For each positive integer $d$ and each nonempty set $T_{n,m}$, the set consisting of all $d$-LMRTTGs 
near $1$ is nonempty and equals $T_{n,m}^{d}(m)$.
\end{lemma}
\begin{proof}
As $T_{n,m}$ is nonempty and a maximum is attained over each finite set, $T_{n,m}^d(1)$ as well as each of the sets $T_{n,m}^d(2),T_{n,m}^d(3),\ldots,T_{n,m}^{d}(m)$ is nonempty. Observe that, by its construction, for each pair of two-terminal graphs $G$ and $G'$ in the set $T_{n,m}^d(m)$ and each $i\in \{1,2,\ldots,m\}$ it holds that $N_i^d(G)=N_i^d(G')$, thus equation~\eqref{equation:poly} gives that $R_{G}^{d}(\rho)=R_{G'}^{d}(\rho)$. It suffices to prove that for each $G$ in $T_{n,m}^d(m)$ and each $H$ in $T_{n,m}-T_{n,m}^d(m)$ there exists $\delta>0$
 such that for every $\rho\in (1-\delta,1)$ it holds that $R_G^d(\rho)>R_H^d(\rho)$. 
 As $H$ is not in $T_{n,m}^d(m)$, 
there exists $k$ in $\{1,2,\ldots,m\}$ such that $N_i^d(G)=N_i^d(H)$ for each $i\in \{1,2,\ldots,k-1\}$ but $N_k^d(G)>N_k^d(H)$. The statement then follows from Remark~\ref{remark:local}\ref{it1:local}.
\end{proof}

We will sometimes construct the sets $T_{n,m}^{d}(1),T_{n,m}^{d}(2),\ldots,T_{n,m}^{d}(m)$ and conclude that the set $T_{n,m}^{d}(m)$ is a single two-terminal graph $G$. In those cases a dichotomy arises:  either $G$ is $d$-UMRTTG in $T_{n,m}$ or no $d$-UMRTTG exists in $T_{n,m}$. This dichotomy is presented in Remark~\ref{remark:uniqueness}.

\begin{remark}\label{remark:uniqueness}
Assume that $T_{n,m}^d(m)=\{G\}$. Lemma~\ref{lemma:existenceLMRTTG} gives that for each $H$ in $T_{n,m}-\{G\}$ there exists $\delta>0$ such that for every $\rho$ in $(1-\delta,1)$ it holds that $R_G^d(\rho)>R_H^d(\rho)$. Therefore, either $G$ is the only $d$-UMRTTG in $T_{n,m}$ or there exists $\rho_0$ in $(0,1)$ and $H$ in $T_{n,m}$ such that $R_{H}^d(\rho_0)>R_G^d(\rho_0)$, in which case no $d$-UMRTTG exists in $T_{n,m}$.     
\end{remark}

The existence or nonexistence of $d$-UMRTTGs has only been studied in the literature when the distance constraint is dropped, that is, when $d\geq n-1$. Throughout the remainder of the article, when $d\geq n-1$, we omit the superscript $d$ and write LMRTTG, UMRTTG, $N_i(G)$, and $R_G(\rho)$ instead of $d$-LMRTTG, $d$-UMRTTG, $N_i^d(G)$, and $R_G^d(\rho)$. The study of the existence or nonexistence of UMRTTGs was pioneered by Bertrand et al.~\cite{Bertrand-2018} and by Xie et al.~\cite{Xie-2021}. In the remainder of this section we will give a list of results appearing in those works that will be useful for our purposes. First, a couple of definitions will be essential to understand such results. 

\begin{definition}
Let $n$ and $r$ be integers such that $n\geq 3$ and $r\in \{0,1,\ldots,n-3\}$. The two-terminal graph $A_{n,r}$ has vertex set $\{s,t\}\cup\{v_3,v_4,\ldots,v_n\}$ and edge set $\{st\} \cup \{sv_i,v_it, i\in \{3,4,\ldots,n\}\} \cup \{v_3v_{3+j}, j\in \{1,2,\ldots,r\}\}$.
\end{definition}
The two-terminal graph $A_{n,r}$ is depicted in Figure~\ref{anr}. 
\begin{definition}
Let $n$ and $m$ be integers such that $n\geq 5$ and $5\leq m \leq 2n-3$. The two-terminal graph $H_{n,m}$ has vertex set $\{s,t\}\cup \{v_3,v_4,\ldots,v_n\}$ and edge set $\{st\}\cup \{sv_i,v_it,i\in \{3,4,\ldots, \frac{m-1}{2}\}\}$ when $m$ is odd. If $m$ is even then the extra edge $v_3v_4$ is also added.
\end{definition}

Among all two-terminal graphs including the edge $st$ joining its terminals, $H_{n,m}$ has the maximum number of paths with $2$ edges joining its terminals. Bertrand et al.~\cite{Bertrand-2018} used the construction stated in Lemma~\ref{lemma:existenceLMRTTG} to determine the only LMRTTG near $1$ in $T_{n,m}$ for some pairs of positive integers $n$ and $m$.

\begin{figure}
    \centering
\begin{tikzpicture}
\filldraw[black] (4.5,2) circle (2pt); 
\node[above] (s) at (4.5,2) {$v_2=t$}; 

\filldraw[black] (0,0) circle (2pt); 
\node[below] at (0,0) {$v_3$}; 
\filldraw[black] (1,0) circle (2pt); 
\node[below] at (1,0) {$v_4$}; 
\filldraw[black] (2,0) circle (2pt); 
\node[below] at (2,0) {$v_5$}; 

\filldraw[black] (2.6,0) circle (.5pt); 
\filldraw[black] (3,0) circle (.5pt); 
\filldraw[black] (3.4,0) circle (.5pt); 

\filldraw[black] (3.6,1) circle (.5pt); 
\filldraw[black] (3.8,1) circle (.5pt); 
\filldraw[black] (4,1) circle (.5pt); 
\filldraw[black] (3.6,-1) circle (.5pt); 
\filldraw[black] (3.8,-1) circle (.5pt); 
\filldraw[black] (4,-1) circle (.5pt); 

\filldraw[black] (4,0) circle (2pt); 
\node[below] at (4,0) {$v_{r+2}$}; 
\filldraw[black] (5,0) circle (2pt); 
\node[below] at (5,0) {$v_{r+3}$}; 
\filldraw[black] (6,0) circle (2pt); 
\node[below] at (6,0) {$v_{r+4}$}; 
\filldraw[black] (7,0) circle (2pt); 
\node[below] at (7,0) {$v_{r+5}$}; 

\filldraw[black] (7.6,0) circle (.5pt); 
\filldraw[black] (8,0) circle (.5pt); 
\filldraw[black] (8.4,0) circle (.5pt); 

\filldraw[black] (6,1) circle (.5pt); 
\filldraw[black] (6.2,1) circle (.5pt); 
\filldraw[black] (6.4,1) circle (.5pt); 
\filldraw[black] (6,-1) circle (.5pt); 
\filldraw[black] (6.2,-1) circle (.5pt); 
\filldraw[black] (6.4,-1) circle (.5pt); 

\filldraw[black] (9,0) circle (2pt); 
\node[below] at (9,0) {$v_{n}$}; 

\filldraw[black] (4.5,-2) circle (2pt); 
\node[below] (s) at (4.5,-2) {$v_1=s$}; 

\draw (4.5,2) -- (0,0);
\draw (4.5,2) -- (1,0);
\draw (4.5,2) -- (2,0);
%\draw (4.5,1) -- (3,0);
\draw (4.5,2) -- (4,0);
\draw (4.5,2) -- (5,0);
\draw (4.5,2) -- (6,0);
\draw (4.5,2) -- (7,0);
\draw (4.5,2) -- (9,0);

\draw (4.5,-2) -- (0,0);
\draw (4.5,-2) -- (1,0);
\draw (4.5,-2) -- (2,0);
%\draw (4.5,1) -- (3,0);
\draw (4.5,-2) -- (4,0);
\draw (4.5,-2) -- (5,0);
\draw (4.5,-2) -- (6,0);
\draw (4.5,-2) -- (7,0);
\draw (4.5,-2) -- (9,0);

\draw (0,0) -- (1,0);
\draw (0,0) to[bend left] (2,0);
\draw (0,0) to[bend left] (4,0);
\draw (0,0) to[bend left] (5,0);

\filldraw[black] (1.2,.35) circle (.5pt); 
\filldraw[black] (1.3,.4) circle (.5pt); 
\filldraw[black] (1.4,.45) circle (.5pt); 

\draw (4.5,2) to[out=0,in=90] (9.5,0);
\draw (9.5,0) to[out=-90,in=0] (4.5,-2);
\end{tikzpicture}
    \caption{\centering Two-terminal graph $A_{n,r}$}
    \label{anr}
\end{figure}
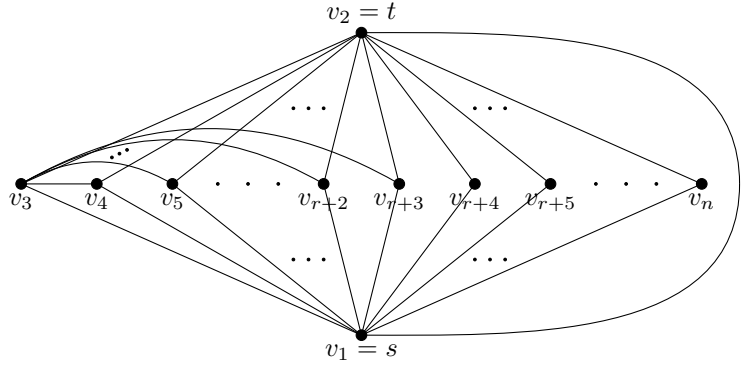

\begin{lemma}[Bertrand et al.~\cite{Bertrand-2018}]\label{lemma:LMRTTG1}
Let $n$ and $m$ be integers such that $n\geq 5$ and $m\geq 5$. The following assertions hold.
\begin{enumerate}[label=(\roman*)]
\item\label{item1} If $m\leq 2n-3$ then $H_{n,m}$ is the only LMRTTG near $1$ in $T_{n,m}$.
\item\label{item2} If $2n-3<m\leq 3n-6$ then $A_{n,m-2n+3}$ is the only LMRTTG near $1$ in $T_{n,m}$. 
\end{enumerate}
\end{lemma}

A significant advance in the study of LMRTTGs near $1$ was given by Gong and Lin~\cite{Gong-2024}. They succeeded in characterizing all LMRTTGs near $1$ in a large number of nonempty classes $T_{n,m}$ such that $n\geq 6$ and $3n-6\leq m\leq \binom{n}{2}$; see~\cite{Gong-2024} for further details. 

Bertrand et al.~\cite{Bertrand-2018} succeeded in finding, for several pairs of integers $n$ and $m$ meeting the conditions of Lemma~\ref{lemma:LMRTTG1}, a two-terminal graph in $T_{n,m}$ whose two-terminal reliability is greater than that of the only LMRTTG near $1$ when $\rho=1/2$. In a first step, they derived an expression for $R_{A_{n,r}}(\rho)$.
\begin{lemma}[Bertrand et al.~\cite{Bertrand-2018}]
For each pair of integers $n$ and $r$ such that 
$n\geq 5$ and $r\in \{0,1,\ldots,n-3\}$ and every $\rho$ in $[0,1]$,
\begin{equation}\label{eq:Anr}
R_{A_{n,r}}(\rho) = 1 + \rho^{n+r}(3-2\rho)^r(2-\rho)^{n-r-3}-2\rho^{n-1}(1+\rho-\rho^2)^r(2-\rho)^{n-r-3}.
\end{equation} 
\end{lemma}

For each $G$ in $T_{n,m}$ 
we denote by $F(G)$ the number of spanning graphs of $G$ 
that do not have a path joining its terminals. 
Observe that $F(G)=2^m-\sum_{i=1}^{m}N_i(G)$. 
Then, equation~\eqref{equation:poly} gives that 
\begin{equation*}
R_G(1/2)=\sum_{i=1}^{m}\frac{N_i(G)}{2^m} = 1-\frac{F(G)}{2^m}.    
\end{equation*}
Then, $F(G)=2^m(1-R_{G}(1/2))$, and $F(A_{n,r})$ can be derived using equation~\eqref{eq:Anr}. 
\begin{equation}\label{eq:F}
F(A_{n,r})=3^{n-r-3}\times (4\times 5^{r}-4^{r}).  
\end{equation}

Bertrand et al. constructed a two-terminal graph $G_{n,m}$ such that $F(G_{n,m})<F(H_{n,m})$ when $m\leq 2n-3$ and $F(G_{n,m})<F(A_{n,m-2n+3})$ when $2n-3<m\leq 3n-9$, 
thus concluding the following result. 
\begin{lemma}[Bertrand et al.~\cite{Bertrand-2018}]\label{lemma:half}
Let $n$ and $m$ be integers such that $n\geq 11$ and $20\leq m \leq 3n-9$. Define $n'=\lceil \frac{n}{3}\rceil +2$, $r'=m-2n'+3$, and the two-terminal graph $G_{n,m}$ in $T_{n,m}$ given by $A_{n',r'}\cup \overline{K_{n-n'}}$. The following assertions hold:
\begin{enumerate}[label=(\roman*)]
\item\label{half1} If $m\leq 2n-3$ then $R_{G_{n,m}}(1/2)>R_{H_{n,m}}(1/2)$.
\item\label{half2} If $2n-3<m\leq 3n-9$ then $R_{G_{n,m}}(1/2) > R_{A_{n,m-2n+3}}(1/2)$.
\end{enumerate}
\end{lemma}

Consequently, Bertrand et al. resolved the dichotomy presented in Remark~\ref{remark:uniqueness} thus proving the nonexistence of UMRTTGs in $T_{n,m}$ for each of the pairs $n$ and $m$ meeting the conditions of Lemma~\ref{lemma:half}. The authors summarize their findings in Theorem~\ref{theorem:bertrand}.
\begin{theorem}[Bertrand et al.~\cite{Bertrand-2018}]\label{theorem:bertrand}
For each pair of integers $n$ and $m$ such that $n\geq 11$ and $20 \leq m \leq 3n-9$ no UMRTTG exists in $T_{n,m}$.
\end{theorem}

An elegant result on the nonexistence of UMRTTGs was given by Xie et al.~\cite{Xie-2021}. In the following, for each $G$ in $T_{n,m}$ we denote by $\hat{G}$ the simple graph that arises from $G$ by removing both of its terminals. 
\begin{theorem}[Xie et al.~\cite{Xie-2021}]\label{theorem:Xie}
Let $n$ and $m$ be integers such that $n\geq 11$ and $3n-6<m\leq \binom{n}{2}-2$. Let $G$ be any two-terminal graph in $T_{n,m}$. The following assertions hold.
\begin{enumerate}[label=(\roman*)]
\item\label{teo-regular} If $G$ is a LMRTTG near $0$ then its  terminals are universal and $\hat{G}$ is almost-regular. 
\item\label{teo-not-regular} If $G$ is a LMRTTG near $1$ then its  terminals are universal and $\hat{G}$ is not almost-regular.
\end{enumerate}
In particular, no UMRTTG exists in $T_{n,m}$ when $n\geq 11$ and $3n-6<m\leq \binom{n}{2}-2$.
\end{theorem}

We briefly describe the method employed by Xie et al. to prove Theorem~\ref{theorem:Xie}.
Let $n$ and $m$ be integers such that $n\geq 11$ and $3n-6< m \leq \binom{n}{2}-2$. Let $T_{n,m}^{u}$ be the collection of all two-terminal graphs in $T_{n,m}$ such that both of its terminals are universal. 

On the one hand, the authors proved that if $G$ is a LMRTTG near $0$ then $G$ is in $T_{n,m}^u$ and the simple graph $\hat{G}$ has the maximum edge-connectivity. It is known~\cite{Harary-1962} that if a graph has the maximum edge-connectivity then it is almost-regular, which gives Theorem~\ref{theorem:Xie}\ref{teo-regular}.  

On the other hand, the authors proved that $T_{n,m}(3)$ equals the set $T_{n,m}^u$. 
Then, the authors derived the following expression for $N_4(G)$ for each two-terminal graph $G$ in $T_{n,m}^{u}$:
\begin{equation*}
N_4(G)= \binom{m-1}{3} + (n-2)\binom{m-3}{2}-\binom{n-2}{2}+2(m-2n+3)(m-6)+2p_3(\hat{G}). 
\end{equation*}
Therefore, a two-terminal graph $G$ is in $T_{n,m}(4)$ if and only if both of its terminals are universal and the simple graph $\hat{G}$ is an almost-regular graph maximizing the number of subgraphs that are isomorphic to $P_3$, that we denote by $p_3(\hat{G})$. Byer~\cite{Byer-1999} proved that every simple graph $\hat{G}$ maximizing $p_3(\hat{G})$ belongs to a subclass of threshold graph. Therefore, Xie et al.~\cite{Xie-2021} proved that none of the threshold graphs found by Byer is almost-regular. Consequently, the set $T_{n,m}(4)$ consists of some two-terminal graphs $G$ in $T_{n,m}^u$ such that $\hat{G}$ is not almost-regular. As $T_{n,m}(m) \subseteq T_{n,m}(4)$, 
Theorem~\ref{theorem:Xie}\ref{teo-not-regular} then follows from Lemma~\ref{lemma:existenceLMRTTG}.

Observe that if $m=\binom{n}{2}$ then $T_{n,m}$ consists of the complete graph $K_n$ equipped with two terminals, which is UMRTTG. Now, if 
$m=\binom{n}{2}-1$ then $T_{n,m}$ consists of precisely $3$ non-isomorphic two-terminal graphs here called $X_n$, $Y_n$ and $Z_n$, 
where $X_n=K_n-st$, $Y_n=K_n-sv_3$, and $Z_n=K_n-v_3v_4$. 
\begin{theorem}\label{theorem:almostcomplete1}
For each integer $n$ such that $n\geq 4$, 
the only UMRTTG in $T_{n,\binom{n}{2}-1}$ 
is the two-terminal graph $Z_n$ defined as $K_n-v_3v_4$.
\end{theorem}

We remark that Bertrand et al.~\cite{Bertrand-2018} and Xie et al.~\cite{Xie-2021} determined the existence or nonexistence of UMRTTGs in some nonempty sets $T_{n,m}$ such that $n\leq 10$ as well, but these cases can be obtained computationally. A consequence of Theorems~\ref{theorem:bertrand},~\ref{theorem:Xie} and~\ref{theorem:almostcomplete1} is that the only cases not covered in previous research occur when $m \in \{3n-8,3n-7,3n-6\}$. Here it is proved that 
no UMRTTG exists in $T_{n,m}$ when $m=3n-8$ and $n\geq 11$ (see Proposition~\ref{prop:3n-8}). A discussion of the existence of $d$-UMRTTG in each nonempty set $T_{n,m}$ as a function of the distance $d$ is also given.

We close this section with a technical lemma from majorization theory~\cite{Marshall-2011} which will be convenient for our purposes.  Define $D_{n,m} = \{(x_1,\dots,x_n) \in \mathbb{Z}_{\geq 0}^n : \sum_{i=1}^n x_i = 2m\}$. A tuple $(x_1,\dots,x_n)$ in $D_{n,m}$ is \emph{fair} if $|x_i - x_j| \leq 1$ for all $i,j\in \{1,2,\ldots,n\}$. 
\begin{lemma}[Marshall et al.~\cite{Marshall-2011}]\label{lemma:Schur}
Let $f: D_{n,m} \to \mathbb{R}$ be the function $f(d_1,\dots,d_n) = \sum_{i=1}^n \phi(d_i)$, where $\phi:\mathbb{R}\to \mathbb{R}$ is a convex function. Then $f$ attains its minimum in $D_{n,m}$ precisely at fair tuples in $D_{n,m}$.
\end{lemma}
Observe that the set $D_{n,m}$ strictly includes all possible tuples consisting of degree sequences of simple graphs on $n$ vertices and $m$ edges, while the fair tuples in $D_{n,m}$ consist precisely of the degree sequences of almost-regular graphs.
% \begin{proof}[Sketch]
% Suppose $(d_1,\dots,d_n)$ is not fair. Then there exist $i,j$ with $d_i \ge d_j + 2$. Consider the pair $(d_i,d_j)$ and replace it by $(d_i-1,d_j+1)$, keeping the sum fixed.  

% By classical convexity of $\phi$ on $\mathbb{R}$, we have
% \[
% \phi(d_i-1)+\phi(d_j+1) \le \phi(d_i)+\phi(d_j),
% \]
% so $f$ does not increase. Iterating this process eventually leads to a fair tuple, which therefore attains the minimum.
% \end{proof}

\section{Existence of 1-UMRTTGs and 2-UMRTTGs}\label{section:1and2}
Here we will characterize all $1$-UMRTTGs and $2$-UMRTTGs. A key concept is that of $d$-irrelevant edges. 
\begin{definition}\label{definition:irrelevance}
Let $G$ be in $T_{n,m}$ and let $d$ be a positive integer. 
An edge $e$ in $G$ is \emph{$d$-irrelevant} if 
$e$ does not belong to any path of length at most $d$ joining its terminals.
\end{definition}

Lemma~\ref{lemma:irrelevant} follows from the previous concept of $d$-irrelevant edges and equation~\eqref{equation:poly}. 
\begin{lemma}\label{lemma:irrelevant}
Let $G$ be in $T_{n,m}$ and let $d$ be a positive integer. 
If an edge $e$ in $G$ is $d$-irrelevant then $R_{G-e}^d(\rho)=R_G^d(\rho)$ for every $\rho\in [0,1]$. 
\end{lemma}
\begin{proof}
Let $G$ be any two-terminal graph in $T_{n,m}$ and let $d$ be any positive integer. Since $e$ is $d$-irrelevant, it does not belong to any path of length at most $d$ joining the terminals of $G$. Hence, every $d$-pathset of $G$ is either a $d$-pathset of $G-e$, or it contains $e$ and becomes a $d$-pathset of $G-e$ after deleting $e$. 
Therefore, $N_i^d(G)=N_i^d(G-e)+N_{i-1}^d(G-e)$ for each $i\in \{1,2,\ldots,m\}$, with the convention $N_0^d(G-e)=0$. Substituting into~\eqref{equation:poly} and reindexing gives that, for every $\rho$ in $[0,1]$,
\begin{align*}
R_{G}^d(\rho) &= \sum_{i=1}^{m}N_i^d(G)(1-\rho)^i\rho^{m-i}   
    = \sum_{i=1}^{m}N_{i-1}^d(G-e)(1-\rho)^i\rho^{m-i} + 
      \sum_{i=1}^{m-1}N_i^d(G-e)(1-\rho)^i\rho^{m-i}\\
    &= (1-\rho)\sum_{i=1}^{m}N_{i-1}^d(G-e)(1-\rho)^{i-1}\rho^{(m-1)-(i-1)} + \rho \sum_{i=1}^{m-1}N_i^d(G-e)(1-\rho)^{i}\rho^{(m-1)-i}\\
     &= (1-\rho)R^d_{G-e}(\rho) + \rho R_{G-e}^d(\rho)=R^d_{G-e}(\rho). \qedhere
\end{align*} 
\end{proof}

A characterization of all $1$-UMRTTGs is a consequence of Lemma~\ref{lemma:irrelevant}. 

\begin{corollary}\label{cor:d1}
Let $n$ and $m$ be integers such that $n\geq 2$ and $1\leq m \leq \binom{n}{2}$. 
A two-terminal graph $G$ in $T_{n,m}$ is a $1$-UMRTTG if and only if $st$ is in $G$.    
\end{corollary}
\begin{proof}
Let $G$ be any two-terminal graph in $T_{n,m}$, and let $\rho$ be any real number in $[0,1]$. Any edge $e$ such that $e\neq st$ does not belong to a $1$-path joining the terminals, hence it is $1$-irrelevant. 
Then, by Lemma~\ref{lemma:irrelevant}, the edge $e$ can be deleted and the $1$-constrained two-terminal reliability is preserved. On the one hand, if $st$ is not in $G$ then we can remove all edges in $G$ and $R_G^1(\rho)=0$ for every $\rho$ in $[0,1]$. On the other hand, if $st$ is in $G$ 
then $R_G^1(\rho)$ equals the probability that $st$ does not fail, which is $1-\rho$. As $1-\rho\geq 0$, we conclude that $G$ is $1$-UMRTTG if and only if $st$ is in $G$, and the corollary follows.
\end{proof}

Lemma~\ref{lemma:st} roughly states that $st$ is the most relevant edge regardless of the choice of the distance constraint $d$. 
\begin{lemma}\label{lemma:st}
For each $G$ in $T_{n,m}$ such that $st$ is not in $G$, each edge $e$ in $G$, and each positive integer $d$, the two-terminal graph $G-e+st$ is $d$-stronger than $G$.
\end{lemma}
\begin{proof}
Let $G$ be any two-terminal graph in $T_{n,m}$ such that $st$ is not in $G$, and let $e$ be any edge in $G$. 
For every $d$-pathset of $G$, replacing the occurrence of $e$ by $st$ yields a spanning subgraph containing the edge $st$, and thus 
a $d$-pathset of $G-e+st$. As the previous correspondence of $d$-pathsets is one-to-one, there are at least as many $d$-pathsets in $G-e+st$ as $d$-pathsets in $G$ of each size 
thus $N_i^d(G-e+st)\geq N_i^d(G)$ for each $i\in \{1,2,\ldots,m\}$. 
Additionally, the single edge $st$ creates a $1$-pathset in $G-e+st$ that is not in $G$ and $N_1^d(G-e+st)>N_1^d(G)$. Then, 
$G-e+st$ is $d$-stronger than $G$, and the lemma follows.
% Let $G$ be in $T_{n,m}$ such that $st$ is not in $G$ and let $e$ be an edge in $G$. As $st$ is not in $G$ but $st$ is in $G-e+st$, 
% $N_1^d(G)=0$ and $N_1^d(G-e+st)=1$ thus $N_1^d(G-e+st)>N_1^d(G)$. 
% Let $i$ be in $\{1,2,\ldots,m\}$. 
% By Remark~\ref{remark:injectivity}, it suffices to construct an injective function $\varphi_i:\mathcal{P}_i^d(G) \to \mathcal{P}_i^d(G-e+st)$. 
% Let $H$ be in $\mathcal{P}_i^d(G)$. If $e$ is not in $H$ then we define $\varphi_i(H)=H$; otherwise, we define $\varphi_i(H)=H-e+st$. First, we will prove that $\varphi_i(\mathcal{P}_i^d(G)) \subseteq \mathcal{P}_i^d(G-e+st)$. 
% On the one hand, if $H$ is in $\mathcal{P}_i^d(G)$ and $e\notin H$ 
% then $H$ is already a $d$-pathset of $G-e+st$ composed of $i$ edges and $\varphi_i(H)=H$ belongs to $\mathcal{P}_i^d(G-e+st)$. On the other hand, if $H$ is in $\mathcal{P}_i^d(G)$ and $e\in H$ then  $\varphi(H)=H-e+st$ is a $d$-pathset of $G-e+st$ composed of $i$ edges and $\varphi_i(H)=H-e+st$ 
% belongs to $\mathcal{P}_i^d(G-e+st)$. Now, we  prove that $\varphi_i$ 
% is injective. Define $\psi_i:\varphi_i(\mathcal{P}_i^d(G)) \to \mathcal{P}_i^d(G)$ as follows. If $st\notin H$ then $\psi_i(H)=H$; 
% otherwise, we define $\psi_i(H)=H-st+e$. By its definition it holds 
% that $\varphi_i(\psi_i(H))=H$ and $\psi_i(\varphi_i(H))=H$. Therefore, $\psi_i$ is a bijective function thus $\varphi_i$ is bijective when restricted to its codomain, or injective, as required. 
\end{proof}

\begin{lemma}\label{lemma:An0}
For each pair of 
integers $n$ and $d$ such that $n\geq 3$ and $d\geq 2$, the $d$-constrained two-terminal unreliability $U_{A_{n,0}}^d(\rho)$ can be obtained as follows,
\begin{equation}\label{eq:An0}
U_{A_{n,0}}^{d}(\rho) = \rho (1-(1-\rho)^2)^{n-2}.  
\end{equation}
\end{lemma}
\begin{proof}
Recall that $A_{n,0}$ has vertex set $\{s,t\}\cup \{v_3,v_4,\ldots,v_n\}$ and edge set $\{st\}\cup \{sv_i,v_it, i\in \{3,4,\ldots,n\}\}$. Consequently, $A_{n,0}$ consists of $st$ together with $n-2$ edge-disjoint paths of length $2$ joining its terminals. As $d\geq 2$, 
there exists no $d$-path joining $s$ and $t$ if and only if $st$ fails and, simultaneously, each of the $n-2$ paths of length $2$ fails. The probability that $st$ fails equals $\rho$. The probability that 
a single path of length $2$ fails 
equals $1-(1-\rho)^2$. As the edge failures are independent, the $d$-constrained unreliability of $A_{n,0}$ equals $\rho(1-(1-\rho)^2)^{n-2}$ yielding equation~\eqref{eq:An0}. 
\end{proof}

We are in position to characterize all $2$-UMRTTGs.
\begin{theorem}\label{theorem:2UMRTTG}
Let $n$ and $m$ be integers such that $n\geq 3$ and $3\leq m \leq \binom{n}{2}$. A two-terminal graph $G$ in $T_{n,m}$ is a $2$-UMRTTG if and only if $G$ includes $A_{n',0}$ as a subgraph, 
where $n' = \min\{\lfloor \frac{m+3}{2}\rfloor,n\}$.
\end{theorem}
\begin{proof}
Let $K_n$ be the complete graph equipped with terminals $s$ and $t$. 
Every edge of $K_n$ not in a $1$-path or $2$-path joining its terminals is $2$-irrelevant. Hence, by Lemma~\ref{lemma:irrelevant}, $R_{K_n}^2(\rho)=R_{A_{n,0}}^2(\rho)$ for every $\rho\in [0,1]$.

If $m\geq 2n-3$ and $G\in T_{n,m}$ then $R_G^2(\rho)\leq R_{A_{n,0}}^2(\rho)$, 
with equality if and only if $G$ contains $A_{n,0}$ as a subgraph. 
Therefore, $G$ is $2$-UMRTTG if and only if $G$ includes $A_{n,0}$ as a subgraph. 

Assume now that $m<2n-3$. Let $G$ be in $T_{n,m}$ satisfying the conditions of the statement. As $G$ contains $A_{n',0}$ as a subgraph 
where $n'=\lfloor (m+3)/2\rfloor$, every edge outside this subgraph is $2$-irrelevant and $U_G^2(\rho)=U_{A_{n',0}}^2(\rho)$. 
Let $H$ be any two-terminal graph in $T_{n,m}$. By Lemma~\ref{lemma:st}, we may assume without loss of generality that $st\in E(H)$. Deleting all $2$-irrelevant edges from $H$ yields a two-terminal graph isomorphic to $A_{\ell,0}$ for some positive integer $\ell$ and $U_H^2(\rho)=U_{A_{\ell,0}}^2(\rho)$. 
As $A_{\ell,0}$ has $2\ell-3$ edges and $H$ has $m$ edges, it follows that 
$2\ell-3\leq m$ thus $\ell\leq n'$. 
By Lemma~\ref{lemma:An0}, we can find 
$U_{A_{n',0}}^2(\rho)$ as well as $U_{A_{\ell,0}}^2(\rho)$ thus 
\begin{equation*}
U_G^2(\rho)=U_{A_{n',0}}^2(\rho)=\rho(1-(1-\rho)^2)^{n'-2} \leq \rho(1-(1-\rho)^2)^{\ell-2}=U_{A_{\ell,0}}^2(\rho)=U_H^2(\rho),    
\end{equation*}
and equality holds if and only if $\ell=n'$. 
In either case, $G$ is a $2$-UMRTTG if and only if it contains $A_{n',0}$ as a subgraph where $n' = \min\{\lfloor \frac{m+3}{2}\rfloor,n\}$, as required.
\end{proof}

\section{Structural properties of 3-UMRTTGs}\label{section:3}
In this section, we derive structural restrictions satisfied by all $3$-UMRTTGs. 
The structure of $3$-UMRTTGs turns out to depend on the density of the two-terminal graphs. In classes of sparse two-terminal graphs $T_{n,m}$ in which $m\leq 2n-3$, we will see that either $H_{n,m}$ is the only $3$-UMRTTG in $T_{n,m}$ or no $3$-UMRTTG exists. In classes of two-terminal graphs $T_{n,m}$ with more density in which $m> 2n-3$, the terminals of any $3$-UMRTTG $G$ are universal and 
$\hat{G}$ must satisfy regularity or connectivity conditions. We close this section with an exact characterization of the only $d$-UMRTTG in  $T_{n,\binom{n}{2}-1}$ when $d\geq 3$ 
which generalizes Theorem~\ref{theorem:almostcomplete1}. 

\begin{lemma}\label{lemma:3-LMRTTG-near1-low}
Let $n$ and $m$ be integers such that $n\geq 5$ and $5\leq m \leq 2n-3$. The only $3$-LMRTTG near $1$ in $T_{n,m}$ is $H_{n,m}$.
\end{lemma}

\begin{proof}
Let $n$ and $m$ be integers as in the statement. By Lemma~\ref{lemma:existenceLMRTTG}, it suffices to prove that
$T_{n,m}^3(m)=\{H_{n,m}\}$. The set $T_{n,m}^3(1)$ consists of all two-terminal graphs in $T_{n,m}$ containing the edge $st$. The set $T_{n,m}^3(2)$ consists of all two-terminal graphs $G$ containing $st$ together with
$\lfloor \frac{m-1}{2}\rfloor$ edge-disjoint $2$-paths joining its  terminals. If $m$ is odd, then $G$ is isomorphic to $H_{n,m}$ and therefore $T_{n,m}^{3}(2)=\{H_{n,m}\}$. As $T_{n,m}^{3}(m)\subseteq T_{n,m}^{3}(2)$ and $T_{n,m}^{3}(m)$ is nonempty, $T_{n,m}^{3}(m)=\{H_{n,m}\}$. 
If $m$ is even, then there exists an additional edge $e$ in $G$ that is distinct from $st$ and does not belong to any $2$-path joining $s$ and $t$. Observe that the additional edge $e$ creates $3$-paths if and only 
if its endpoints are common neighbors of both terminals $s$ and $t$, a $G$ is isomorphic to $H_{n,m}$. Therefore, 
$T_{n,m}^3(3)=\{H_{n,m}\}$, and again, $T_{n,m}^{3}(m)=\{H_{n,m}\}$. 
In either case $T_{n,m}^3(m)=\{H_{n,m}\}$, and the lemma follows.
\end{proof}

Now we study the existence of $3$-LMRTTGs near $1$ in $T_{n,m}$ when $n\geq 5$ and $2n-3 < m \leq \binom{n}{2}$.
Let $G$ be in $T_{n,m}$. Observe that $G$ belongs to $T_{n,m}^3(2)$ if and only if $st\in E(G)$ and every $2$-path joining $s$ and $t$ is contained in $G$, that is, if and only if $G\in T_{n,m}^u$. Assume now that $G\in T_{n,m}^u$. We derive an expression for $N_i^3(G)$ for each $i\in 
\{3,4,\ldots,m\}$ in terms of the presence or absence of paths joining $s$ and $t$. The main tool is the inclusion-exclusion principle.

Let $v$ be any vertex in $\hat{G}$ and let $xy$ be any edge in $\hat{G}$. Define the collections of edge sets $\mathcal{A}_0$, $\mathcal{A}_v$, $\mathcal{A}_{xy}^x$, and $\mathcal{A}_{xy}^y$ by
\begin{align*}
\mathcal{A}_0 &=\{st\},\\
\mathcal{A}_v &= \{sv,vt\},\\
\mathcal{A}_{xy}^x &= \{sx,xy,yt\},\\
\mathcal{A}_{xy}^y &= \{sy,xy,xt\}.
\end{align*}

A spanning subgraph $H$ of $G$ is a $3$-pathset if and only if it contains the edge $st$, a $2$-path joining the terminals, or a $3$-path joining them. 
The following notation isolates these three possibilities and will allow us to apply the inclusion-exclusion principle. For each $j\in \{1,2,3\}$, define $N_{i,j}(G)$ as the number of spanning subgraphs $H$ of $G$ with precisely $i$ edges such that the distance between $s$ and $t$ in $H$ is precisely $j$, namely,
\begin{equation*}
N_{i,j}(G)=|\{H :H\subseteq G, V(H)=V(G), |E(H)|=i, d_H(s,t)=j\}|.
\end{equation*}

By the sum rule, $N_i^3(G)=N_{i,1}(G)+N_{i,2}(G)+N_{i,3}(G)$. 
By definition, $N_{i,1}(G)$ is the number of spanning subgraphs $H$ of $G$ containing the edge $st$ together with $i-1$ of the remaining $m-1$ edges of $G$, that is,
\begin{equation*}
N_{i,1}(G)=\binom{m-1}{i-1}.
\end{equation*}

Next, $N_{i,2}(G)$ counts all spanning subgraphs of $G$ that do not contain the edge $st$ but contain both edges $sv$ and $vt$ for some vertex $v\in V(\hat{G})$. For each of the $\binom{n-2}{k}$ collections of $k$ edge-disjoint paths of length $2$ joining $s$ and $t$, there are exactly $\binom{m-1-2k}{i-2k}$ ways to choose the additional $i-2k$ edges among the $m-1-2k$ edges of $G$ that are neither $st$ nor part of the chosen paths. Therefore, the inclusion--exclusion principle yields
\begin{equation*}
N_{i,2}(G) = \sum_{k=1}^{\lfloor \frac{i}{2} \rfloor} (-1)^{k+1}\binom{n-2}{k}\binom{m-1-2k}{i-2k}.
\end{equation*}

\begin{remark}\label{remark:choice}
Let $G\in T_{n,m}^u$ and let $i\in \{1,2,\ldots,m\}$. The quantities $N_{i,1}(G)$ and $N_{i,2}(G)$ do not depend on the choice of $G$.
\end{remark}

We now derive an expression for $N_{i,3}(G)$ when $G\in T_{n,m}^u$. Let $\mathcal{A}$ be the set of all spanning subgraphs $H$ of $G$ such that neither $\mathcal{A}_0 \subseteq E(H)$ nor $\mathcal{A}_v \subseteq E(H)$ for any $v\in V(\hat{G})$. Then a spanning subgraph $H$ is in $\mathcal{A}$ if and only if $d_H(s,t)\geq 3$. 
To ensure that $d_H(s,t)=3$, at least one of the sets $\mathcal{A}_{xy}^x$ or $\mathcal{A}_{xy}^y$ must be contained in $H$ for some edge $xy\in E(\hat{G})$. Define the collection of sets $\mathcal{Y}=\{\mathcal{A}_{xy}^x, \mathcal{A}_{xy}^y: xy\in \hat{G}\}$. 
For each subset $\mathcal{X}\subseteq \mathcal{Y}$, let $N_G(\mathcal{A},\mathcal{X})$ be the number of spanning subgraphs $H$ of $G$ belonging to $\mathcal{A}$ and containing every edge set in $\mathcal{X}$. Another application of the inclusion-exclusion principle gives
\begin{equation}\label{eq:incexc3}
N_{i,3}(G)= \sum_{\mathcal{X}\neq \emptyset: \mathcal{X}\subseteq \mathcal{Y}}(-1)^{|\mathcal{X}|+1}N_G(\mathcal{A},\mathcal{X}).
\end{equation}

Lemma~\ref{lemma:3-LMRTTG-near1-high} provides necessary conditions for a two-terminal graph $G$ in $T_{n,m}$ to be $3$-LMRTTG near $1$ when $n\geq 5$ and $m>2n-3$.

\begin{lemma}\label{lemma:3-LMRTTG-near1-high}
Let $n$ and $m$ be integers such that $n\geq 5$ and 
$m > 2n-3$. If $G$ is a $3$-LMRTTG near $1$ in 
$T_{n,m}$, then $G$ belongs to $T_{n,m}^u$ and 
$\hat{G}$ is an almost-regular graph. 
\end{lemma}
\begin{proof}
Let $n$ and $m$ be integers such that $n\geq 5$ and $m>2n-3$. By Lemma~\ref{lemma:existenceLMRTTG}, the set of all $3$-LMRTTGs near $1$ in $T_{n,m}$ is precisely $T_{n,m}^3(m)$. Since $T_{n,m}^3(5)\subseteq T_{n,m}^3(m)$, it suffices to prove that every graph $G\in T_{n,m}^3(5)$ belongs to $T_{n,m}^u$ and satisfies that $\hat{G}$ is an almost-regular graph.

Let $G\in T_{n,m}$. By definition, $G\in T_{n,m}^3(1)$ if and only if $st\in E(G)$, and $G\in T_{n,m}^3(2)$ if and only if $G$ contains $st$ together with every path of length $2$ joining $s$ and $t$. Hence, 
$T_{n,m}^3(2)=T_{n,m}^u$. We will show that $
T_{n,m}^3(3)=T_{n,m}^3(4)=T_{n,m}^u$, and finally that $T_{n,m}^3(5)$ consists precisely of those two-terminal graphs $G\in T_{n,m}^u$ such that $\hat{G}$ is an almost-regular graph. Let $G\in T_{n,m}^u$ and let $i\in \{3,4,\ldots,m\}$. By Remark~\ref{remark:choice}, the quantities $N_{i,1}(G)$ and $N_{i,2}(G)$ do not depend on the choice of $G$. Since $N_i^3(G)=N_{i,1}(G)+N_{i,2}(G)+N_{i,3}(G)$, it follows that $T_{n,m}^3(i+1)$ consists of all two-terminal graphs that maximize $N_{i,3}(G)$ among all graphs in $T_{n,m}^3(i)$. Each of the $m-2n+3$ edges of $\hat{G}$ belongs to exactly two $3$-paths 
 joining $s$ and $t$. Therefore, equation~\eqref{eq:incexc3} yields
\[
N_{3,3}(G)=2(m-2n+3),
\]
which is independent of the choice of $G\in T_{n,m}^u$. Consequently, 
$T_{n,m}^3(3)=T_{n,m}^u$. Now let $H$ be a spanning subgraph of $G$ with exactly $4$ edges containing the three edges of some set $\mathcal{A}_{xy}^{x}$, namely,
$sx,xy,yt\in E(H)$. Each $3$-path contributes exactly $m-6$ admissible choices for the fourth edge since the edges $st$, $xt$, and $sy$ are forbidden. Hence, equation~\eqref{eq:incexc3} gives
\[
N_{4,3}(G)=2(m-2n+3)(m-6).
\]
Again, this quantity does not depend on the choice of $G\in T_{n,m}^u$ thus  $T_{n,m}^3(4)=T_{n,m}^u$. Finally, we determine the set $T_{n,m}^3(5)$. Let $G\in T_{n,m}^u$, and let $H$ be a spanning subgraph of $G$ with exactly $5$ edges. Suppose first that $H$ contains exactly one path of length $3$ joining $s$ and $t$. Then $H$ consists of the edges of either $\mathcal{A}_{xy}^x$ or $\mathcal{A}_{xy}^y$ for some edge $xy\in E(\hat{G})$, together with two additional edges chosen from the $m-6$ edges not belonging to $\{sx,xy,yt,st,xt,sy\}$,
under the condition that these two extra edges do not form a $2$-path joining $s$ and $t$. There are exactly 
$2(m-2n+3)\left(\binom{m-6}{2}-(n-4)\right)$ such choices for $H$. Now suppose that $H$ contains exactly two paths of length $3$ joining $s$ and $t$. Then $H$ contains a pair of adjacent edges $uv$ and $vw$ in $\hat{G}$ together with either all edges in
$\{su,sw,vt\}$ or all edges in $\{sv,ut,wt\}$. 
Since there are exactly $p_3(\hat{G})$ pairs of adjacent edges in $\hat{G}$, there are $2p_3(\hat{G})$ such choices. 
Substituting into equation~\eqref{eq:incexc3}, we obtain that
\begin{equation*}
N_{5,3}(G)=2(m-2n+3)\left(\binom{m-6}{2}-(n-4)\right)-2p_3(\hat{G}).
\end{equation*}
Therefore, $N_{5,3}(G)$ is maximum among all two-terminal graphs $G\in T_{n,m}^u$ if and only if $p_3(\hat{G})$ is minimum. Let $(d_3,d_4,\ldots,d_n)$ be the degree sequence of $\hat{G}$. As $p_3(\hat{G})$ equals the number of pairs of adjacent edges in $\hat{G}$,  $p_3(\hat{G})=\sum_{i=3}^{n}\binom{d_i}{2}$. 
As $\hat{G}$ is a simple graph with $n-2$ vertices and $m-2n+3$ edges, its degree sequence belongs to the set $D_{n-2,m-2n+3}$. Moreover, 
$p_3(\hat{G})=\sum_{i=3}^{n}\phi(d_i)$, where
$\phi(x)=\frac{x(x-1)}{2}$ is a convex real-valued function. By Lemma~\ref{lemma:Schur}, the quantity $p_3(\hat{G})$ is minimized if and only if $d(\hat{G})$ is a fair tuple, or equivalently, if and only if $\hat{G}$ is an almost-regular graph. This completes the proof.
\end{proof}

We are in position to derive necessary conditions for $3$-LMRTTGs near $1$ in each nonempty set $T_{n,m}$. 
\begin{proposition}[Structural restrictions for $d$-LMRTTGs near $1$]\label{proposition:3-LMRTTG-near1}
Let $n$ and $m$ be integers such that $n\geq 5$ and $m\geq 5$. All the following assertions hold.
\begin{enumerate}[label=(\roman*)]
\item If $m\leq 2n-3$ then $H_{n,m}$ is the only $3$-LMRTTG near $1$ in $T_{n,m}$.
\item If $2n-3<m\binom{n}{2}$ and $G$ is $3$-LMRTTG near $1$ in $T_{n,m}$ then $G$ is in $T_{n,m}^u$ and $G-s-t$ is almost-regular. 
 \end{enumerate}
\end{proposition}
\begin{proof}
Assertion (i) follows from Lemma~\ref{lemma:3-LMRTTG-near1-low}, while assertion (ii) follows from Lemma~\ref{lemma:3-LMRTTG-near1-high}.    
\end{proof}

We now study $d$-LMRTTGs near $0$ in $T_{n,m}$ when $d\geq 3$. The analysis naturally splits into classes according to the density of the two-terminal graphs. As we shall see, the structure of  $d$-LMRTTGs is modified   when $m$ crosses the thresholds $2n-3$ and $3n-5$.

Let $G$ be in $T_{n,m}$. An edge set $U \subseteq E(G)$ is a \emph{$d$-cutset of $G$} if $G-U$ has no $d$-pathset. A \emph{trivial $d$-cutset of $G$} is a $d$-cutset $U$ containing all edges incident to one of the terminals of $G$. As it will turn out, the number of 
small $d$-cutsets will determine the structure of 
$d$-LMRTTGs near $0$ in $T_{n,m}$. 
 
For each $i\in \{1,2,\ldots,m\}$ and each positive integer $d$ we define $B_i^d(G)$ as the number of $d$-cutsets of $G$ with precisely $i$ edges. 
By the sum-rule, $B_i^d(G)=\binom{m}{i}-N_{m-i}^d(G)$. 
Let $\lambda_{st}^d(G)$ be the least positive integer $j$ such that $B_j^d(G)$ is positive. We omit the superscript and write $\lambda_{st}(G)$ instead of  $\lambda_{st}^d(G)$ when  $d\geq n-1$. 

\begin{lemma}\label{lemma:3-LMRTTG-near0-low}
Let $n$, $m$ and $d$ be integers such that $n\geq 5$,  $5\leq m \leq 2n-3$ and $d\geq 3$. The only $d$-LMRTTG near $0$ in $T_{n,m}$ is $H_{n,m}$.
\end{lemma}
\begin{proof}
Let $G$ be a $d$-LMRTTG near $0$ in $T_{n,m}$. 
By Remark~\ref{remark:local}\ref{it2:local} it follows that $\lambda_{st}^d(G)$ must be maximum. 

If $m$ is odd then $\lambda_{st}^d(G)\leq (m+1)/2$, with equality if and only if $G$ contains the edge $st$ together with $(m-1)/2$ edge-disjoint $2$-paths joining $s$ and $t$, thus $G$ is isomorphic to $H_{n,m}$. If $m$ is  even, then $\lambda_{st}^d(G) \leq m/2$, and equality holds if and only if precisely one of the following conditions hold (i) $G$ is isomorphic to $H_{n,m}$ or (ii) $G$ consists of $st$ as well as 
$m/2-2$ edge-disjoint $2$-paths joining $s$ and $t$ 
and a path $su,uv,vt$ that is disjoint with the previous paths; let us denote this graph by $H_{n,m}'$. Observe that $B_j^d(H_{n,m})=B_j^d(H_{n,m}')=0$ 
when $j\in \{0,1,\ldots,m/2-1\}$. We will proceed to compare 
$B_{m/2}^d(H_{n,m})$ with $B_{m/2}^d(H_{n,m}')$. 
On the one hand, every $d$-cutset of $H_{n,m}$ consisting of $m/2$ edges includes the edge $st$, either $\{sv_3,sv_4\}$ or $\{v_3t,v_4t\}$, and $1$ edge out of the $2$ edges from each of the $m/2-3$ remaining edge-disjoint $2$-paths, thus $B_{m/2}^d(H_{n,m})=2^{m/2-2}$. 
On the other hand, every $d$-cutset of $H_{n,m}'$ 
consisting of $m/2$ edges includes the edge $st$, $1$ out of the $2$ edges from each of the $m/2-3$ edge-disjoint $2$ paths, and $1$ out of the $3$ edges in $\{su,uv,vt\}$, thus $B_{m/2}^d(H_{n,m}')=3\times 2^{m/2-3}$. Therefore, $B_{m/2}^d(H_{n,m}')>B_{m/2}^d(H_{n,m})$. 
As $G$ is $d$-LMRTTG near $0$ in $T_{n,m}$, Remark~\ref{remark:local}\ref{it2:local} gives  
that $G$ is not isomorphic to $H_{n,m}'$ 
thus it is isomorphic to $H_{n,m}$, as required. 
\end{proof}
Lemmas~\ref{lemma:3-LMRTTG-near1-low} and \ref{lemma:3-LMRTTG-near0-low} respectively assert that $H_{n,m}$ is the only $d$-LMRTTG both near $1$ and $0$ in $T_{n,m}$ when $d\geq 3$. In~\cite{RomeroLAGOS2-2025}, an inductive argument was proposed to show that $H_{n,m}$ is $3$-UMRTTG in $T_{n,m}$. However, one step of the argument relies on the assumption that the $3$-constrained two-terminal reliability of $H_{n,m}$ minus an edge $sv$ coincides with that of $H_{n,m-1}$, which does not hold in general. By Remark~\ref{lemma:minimize:n-2}, either $H_{n,m}$ is $3$-UMRTTG or no such $3$-UMRTTG exists. This is still an open problem. 

We now derive necessary conditions for a two-terminal graph $G$ to be $3$-LMRTTG near $0$ in $T_{n,m}$ when $m> 2n-3$. 
Lemma~\ref{lemma:minimize:n-2} is a consequence of Menger's theorem. 
\begin{lemma}\label{lemma:minimize:n-2}
Let $n$, $m$ and $d$ be integers such that $n\geq 5$, $2n-3 < m\leq \binom{n}{2}$, and $d\geq 3$. Then $B_i^d(G)=0$ 
for all $i\in \{0,1,\ldots,n-2\}$ if and only if $G$ 
is in $T_{n,m}^u$. 
\end{lemma}
\begin{proof}
Let $n$, $m$ and $d$ be as in the statement, and let $G$ be any two-terminal graph in $T_{n,m}$. On the one hand, if $G$ is in $T_{n,m}^u$ then 
$G$ includes $st$ as well as precisely $n-2$ edge-disjoint $2$-paths joining $s$ and $t$. By Menger’s theorem, $\lambda_{st}(G) = n-1$. 
As each of the previous $n-1$ edge-disjoint paths have length $1$ or $2$ and $d\geq 3$ it follows that $\lambda_{st}^d(G)=n-1$, and $B_i^d(G)=0$ for all $i\in \{0,1,\ldots,n-2\}$.   
Conversely, if $G\notin T_{n,m}^u$, then $G$ has a trivial $d$-cutset consisting of at most $n-2$ edges, thus $\lambda_{st}^d(G)\leq n-2$ and $B_{n-2}^d(G)>0$. The lemma follows.
\end{proof}

\begin{lemma}\label{lemma:3-LMRTTG-near0-medium}
Let $n$, $m$ and $d$ be integers such that $n\geq 5$, $2n-3< m\leq 3n-6$, and $d\geq 3$. If $G$ is $d$-LMRTTG near $0$ in $T_{n,m}$ then $G$ is in $T_{n,m}^u$ and $\hat{G}$ is an almost-regular graph all whose components are paths.
\end{lemma}
\begin{proof}
Let $n$, $m$ and $d$ be as in the statement, and 
let $G$ be any $d$-LMRTTG near $0$ in $T_{n,m}$. Recall that 
$B_i^d(G)=\binom{m}{i}-N_{m-i}^d(G)$. 
As $G$ is $d$-LMRTTG near $0$ in $T_{n,m}$, Remark~\ref{it2:local} implies that the numbers $B_1^d(G),B_2^d(G),\ldots,B_{n-2}^d(G)$ must be minimum. 
Therefore, by Lemma~\ref{lemma:minimize:n-2}, the two-terminal graph $G$ is in $T_{n,m}^u$. Now we will find $B_{n-1}(G)$. 
Let $U$ be any $d$-cutset of $G$ consisting of $n-1$ edges. 
To ensure the absence of $1$-paths and $2$-paths, $U$ 
consists precisely of $st$ as well as edge sets $sW \cup tW^c$, where $W \cup W^c=V(\hat{G})$. Additionally, 
$G-U$ does not include $3$-paths joining $s$ and $t$, or equivalently, there is no edge with one endpoint in $W$ and one endpoint in $W^c$. This occurs precisely when $W$ is a union of components of the simple graph $\hat{G}$. Therefore, the number of $d$-cutsets of $G$ with $n-1$ edges equals the number of all subsets of components of $\hat{G}$, and $B_{n-1}^d(G)=2^{\kappa(\hat{G})}$ where $\kappa(\hat{G})$ is the number of components of $G$. 
Consequently, $B_{n-1}^d(G)$ is minimum if and only if $\kappa(\hat{G})$ is minimum, which occurs precisely when $\hat{G}$ is a forest. 

We now find $B_n^d(G)$ among all two-terminal graphs 
$G$ in $T_{n,m}^u$ such that $\hat{G}$ is a forest. 
Again, any $d$-cutset of $G$ consisting of $n$ edges 
includes $st$, all edges in $sW \cup tW^c$ for some vertex set $W$ in $V(\hat{G})$ plus an extra edge $e$. If $e$ is either $sv$ or $vt$ for some vertex $v$ in $\hat{G}$ then $v$ is not incident to any of the terminals, and there are precisely as many choices of $d$-cutsets including the edge $e$ as $2^{\kappa(\hat{G}-v)}$. As $\hat{G}$ is a forest and $v$ is in $V(\hat{G})$, it follows that $\kappa(\hat{G}-v)=\kappa(G)+d_{\hat{G}}(v)-1$, where $d_{\hat{G}}(v)$ is the degree of vertex $v$ in $\hat{G}$. 
As $\hat{G}$ has $n-2$ vertices, 
there are precisely 
$\sum_{v\in V(\hat{G})}2^{\kappa(G)+d_{\hat{G}}(v)-1}$ such $d$-cutsets. Similarly, if $e$ is in $E(\hat{G})$, then there are $2^{\kappa(\hat{G}-e)}$ choices of $d$-cutsets including $e$. As $\kappa(\hat{G}-e)=\kappa(\hat{G})+1$ and $|E(\hat{G})|=m-2n+3$, there are precisely $(m-2n+3)2^{\kappa(G)+1}$ such $d$-cutsets. Therefore,
\begin{equation*}
B_n^d(G)= \sum_{v\in V(\hat{G})}2^{\kappa(\hat{G})+d_{\hat{G}}(v)-1} + (m-2n+3)2^{\kappa(\hat{G})+1}.  
\end{equation*}
As $\hat{G}$ is a forest on $n-2$ vertices and $m-2n+3$ edges, $\kappa(\hat{G})=n-2-(m-2n+3)=3n-5-m$. Rewriting the previous equation yields
\begin{equation*}
B_n^d(G) = 2^{\kappa(\hat{G})-1}(\sum_{v\in V(\hat{G})}2^{d_{\hat{G}}(v)} + (m-2n+3)) 
= 2^{3n-6-m}(\sum_{v\in V(\hat{G})}2^{d_{\hat{G}}(v)} + (m-2n+3)),
\end{equation*}
thus $B_n^d(G)$ is minimum among two-terminal graphs $G$ in $T_{n,m}^u$ such that $\hat{G}$ is a forest if and only if its degree sequence $(d_3,d_4,\ldots,d_n)$ attains the minimum of the function $f(d_3,d_4,\ldots,d_n)=\sum_{i=3}^{n}2^{d_i}$. 
The tuple $(d_3,d_4,\ldots,d_n)$ is the degree sequence of $\hat{G}$ which has $m-2n+3$ edges; 
therefore it belongs to the set $D_{n-2,m-2n+3}$. 
Further, the function $f$ can be written 
as $f(d_3,d_4,\ldots,d_n)=\sum_{i=3}^{n}\phi(d_i)$, 
where $\phi(x)=2^x$. As $\phi$ is a convex real-valued function, $f$ is minimum  when $(d_3,d_4,\ldots,d_n)$ is a fair tuple in $D_{n-2,m-2n+3}$, which is attained precisely when $\hat{G}$ is an almost-regular graph. As $\hat{G}$ is a forest with at least one edge, it has a vertex 
with degree $1$. 
Since $\hat{G}$ is an almost-regular graph, every vertex has degree either $0$ or $1$, or every vertex has degree either $1$ or $2$. As $\hat{G}$ is a forest, each of its components is a path, and the lemma follows.  
\end{proof}

\begin{lemma}\label{lemma:minimize:n-2+k}
Let $n$, $m$ and $d$ be integers such that $n\geq 5$, $3n-5\leq m\leq \binom{n}{2}$, and $d\geq 3$. 
Let $G$ be in $T_{n,m}^u$. Then, for each 
$i\in \{1,2,\ldots,\lambda(\hat{G})\}$ it holds that $B_{n-2+i}^d(G)= 2\binom{m-n+1}{i-1}$ and $B_{n-1+\lambda(\hat{G})}^d>2\binom{m-n+1}{\lambda(\hat{G})}$. 
\end{lemma}
\begin{proof}
Let $G\in T_{n,m}^u$ and let 
$i\in \{1,2,\ldots,\lambda(\hat{G})\}$. 
Each trivial $d$-cutset with $n-2+i$ edges consists of all $n-1$ edges incident to one terminal together with $i-1$ additional edges, thus 
$B_{n-2+i}^d(G)\geq 2\binom{m-n+1}{i-1}$.

Suppose that $U'$ is a nontrivial $d$-cutset with $n-2+i$ edges. Since $G\in T_{n,m}^u$, the set $U'$ must contain $st$, all edges in $sW\cup tW^c$ for some nonempty proper subset $W\subset V(\hat{G})$, and an edge set $U\subseteq E(\hat{G})$. Hence
$U'=\{st\}\cup sW\cup tW^c\cup U$, where 
$|U|=|U'|-(n-1)=i-1<\lambda(\hat{G})$. Therefore, $\hat{G}-U$ is connected, so there exists an edge $xy$ in $\hat{G}-U$ with $x\in W$ and $y\in W^c$. Consequently, $sx,xy,yt$ is a $3$-path in $G-U'$, a contradiction. Thus every $d$-cutset with $n-2+i$ edges is trivial, and 
$B_{n-2+i}^d(G)=2\binom{m-n+1}{i-1}$.

Now let $i=\lambda(\hat{G})+1$. Since $\lambda(\hat{G})$ is the edge-connectivity of $\hat{G}$, there exists an edge set $U'\subseteq E(\hat{G})$ with $|U'|=\lambda(\hat{G})$ such that $\hat{G}-U'$ is disconnected. Let $W$ be the vertex set of a component of $\hat{G}-U'$, and define
$U=\{st\}\cup sW\cup tW^c\cup U'$. Then $U$ is a nontrivial $d$-cutset with $n-1+\lambda(\hat{G})$ edges. Since there are also $2\binom{m-n+1}{\lambda(\hat{G})}$ trivial $d$-cutsets of this cardinality, it follows that
$B_{n-1+\lambda(\hat{G})}^d(G)>
2\binom{m-n+1}{\lambda(\hat{G})}$. The lemma follows.
\end{proof}

Lemma~\ref{lemma:3-LMRTTG-near0-high} derives a necessary condition for a two-terminal graph in $T_{n,m}$ to be a $d$-LMRTTG near $0$ when $n\geq 5$, $3n-5 \leq m \leq \binom{n}{2}$, and $d\geq 3$. 

\begin{lemma}\label{lemma:3-LMRTTG-near0-high}
Let $n$, $m$ and $d$ be integers such that $n\geq 5$, $3n-5 \leq m \leq \binom{n}{2}$ and $d\geq 3$. 
If $G$ is a $d$-LMRTTG near $0$ then $G$ is in $T_{n,m}^u$ and $\hat{G}$ has the maximum edge-connectivity.
\end{lemma}
\begin{proof}
Let $G$ be a $d$-LMRTTG near $0$ in $T_{n,m}$. Define 
$B^d(G)=(B_0^d(G),\ldots,B_m(G)^d(G))$. 
By Remark~\ref{remark:local}\ref{it2:local}, the tuple 
$B^d(G)=(B_0^d(G),B_1^d(G),\ldots,B_m(G)^d(G))$ is lexicographically minimum among all two-terminal graphs in $T_{n,m}$. Therefore, Lemma~\ref{lemma:minimize:n-2} implies that $G$ is in $T_{n,m}^u$, 
while Lemma~\ref{lemma:minimize:n-2+k} gives that $\lambda(\hat{G})$ must be maximum. The lemma follows.
\end{proof}

We are in position to derive necessary conditions for $3$-LMRTTGs near $0$ in each nonempty set $T_{n,m}$. 
\begin{proposition}[Structural restrictions for $d$-LMRTTGs near $0$]\label{proposition:d-LMRTTG-near0}
Let $n$, $m$ and $d$ be integers such that $n\geq 5$, $m\geq 5$ and $d\geq 3$. Let $G$ be any $d$-LMRTTG near $0$ in $T_{n,m}$. 
All the following assertions hold.
\begin{enumerate}[label=(\roman*)]
\item If $m\leq 2n-3$ then $G$ is isomorphic to $H_{n,m}$.
\item If $2n-3<m \leq 3n-6$ then $G$ is in $T_{n,m}^u$ and $\hat{G}$ is an almost-regular graph all whose components are paths.
\item If $3n-6<m\leq \binom{n}{2}$ then $G$ is in $T_{n,m}^u$ and $\lambda(\hat{G})$ is maximum. 
 \end{enumerate}
\end{proposition}
\begin{proof}
The first assertion follows from Lemma~\ref{lemma:3-LMRTTG-near0-low}, whereas the second and third assertions follow from Lemmas~\ref{lemma:3-LMRTTG-near0-medium} and~\ref{lemma:3-LMRTTG-near0-high}, respectively.  
\end{proof}
Observe that Propositions~\ref{proposition:3-LMRTTG-near1} and~\ref{proposition:d-LMRTTG-near0}
impose compatible structural restrictions on any
$3$-UMRTTG in $T_{n,m}$. We close this section with an exact characterization of the only $d$-UMRTTG in $T_{n,\binom{n}{2}-1}$ when $d\geq 3$. 
\begin{theorem}\label{theorem:almostcomplete}
Let $n$ and $d$ be integers such that $n\geq 4$ and $d\geq 3$. 
Then the two-terminal graph $Z_n$ defined by $K_n-v_3v_4$ is the only $d$-UMRTTG in $T_{n,\binom{n}{2}-1}$.  
\end{theorem}
\begin{proof}
Let $n$ and $d$ be integers such that $n\geq 4$ and $d\geq 3$. Recall that $T_{n,\binom{n}{2}-1}=\{X_n,Y_n,Z_n\}$, 
where $X_n=K_n-st$, $Y_n=K_n-sv_3$, and $Z_n=K_n-v_3v_4$. By Lemma~\ref{lemma:dstrongimpliesdUMRTTG}, it suffices to prove that $Z_n$ is $d$-stronger than both $X_n$ and $Y_n$. As $st$ is not in $X_n$ but $v_3v_4$ is in $X_n$, Lemma~\ref{lemma:st} 
implies that $Z_n=X_n-v_3v_4+st$ is $d$-stronger than $X_n$. It suffices to prove that $Z_n$ is $d$-stronger than $Y_n$. 
As $Z_n$ includes all $2$-paths in $Y_n$ as well as the path $sv_3,v_3t$, 
we obtain that $N_2^d(Z_n)=N_2^d(Y_n)+1$. Let $i \in \{1,2,\ldots,\binom{n}{2}-1\}$. Define, for each $G$ in $T_{n,m}$, the family $\mathcal{P}_i^d(G)$ of all $d$-pathsets of $G$ with $i$ edges. 
Now, define the function $\varphi_i: \mathcal{P}_i^d(Y_n) \to \mathcal{P}_i^d(Z_n)$ 
as follows,
\begin{equation*}
\varphi_i(H) = 
\begin{cases}
        H, \text{ if } v_3v_4 \notin H\\
        H-v_3v_4+sv_3, \text{ if } v_3v_4 \in H   \text{ and } H-v_3v_4+sv_3 \in \mathcal{P}_{i}^d(Z_n)\\
        H-v_3v_4+sv_3-\{sv_i: sv_i\in H\}+\{v_4v_i:sv_i\in H\},      
        \text{ if otherwise}.       
\end{cases}        
\end{equation*}

It suffices to prove that $\varphi_i(\mathcal{P}_i^d(Y_n))\subseteq \mathcal{P}_i^d(Z_n)$ and that $\varphi_i$ is injective.
In fact, this implies that $N_i^d(Y_n)=|\mathcal{P}_i^d(Y_n)|= |\varphi_i(\mathcal{P}_i^d(Y_n))|
\leq |\mathcal{P}_i^d(Z_n)|=N_i^d(Z_n)$, and the theorem follows since $N_2^{d}(Y_n)<N_2^d(Z_n)$. First, we will prove that $\varphi_i(\mathcal{P}_i^d(Y_n))$ is in $\mathcal{P}_i^d(Z_n)$. Let $H$ be any two-terminal graph in  $\mathcal{P}_i^d(Y_n)$. Observe that, in each case, $\varphi_i(H)$ 
has precisely as many edges as $H$ and all edges in 
$\varphi_i(H)$ belong to $Z_n$. Therefore, $\varphi_i(H)$ 
is a spanning subgraph of $Z_n$ with precisely $i$ edges. If either $v_3v_4\notin H$ or $v_3v_4\in H$ 
but $H-v_3v_4+sv_3\in \mathcal{P}_i^d(Z_n)$, then $\varphi(H)$ is in $\mathcal{P}_i^d(Z_n)$. 
If otherwise, $v_3v_4\in H$ and $H-v_3v_4+sv_3 \notin \mathcal{P}_i^d(Z_n)$. 
Let $P$ be any path in $H$ with at most $d$ edges joining $s$ and $t$. As $H-v_3v_4+sv_3$ is not in $\mathcal{P}_{i}^d(Z_n)$, the edge $v_3v_4$ is in $P$ 
and $P-v_3v_4$ is the union of 
two paths $P_1$ and $P_2$, 
where the endpoints of $P_1$ 
are $s$ and $v_3$ while the endpoints of $P_2$ are $v_4$ and $t$. Let $v_i$ be the only vertex in $P_1$ that is adjacent to $s$, and let $P_1'$ be the subpath in $P_1$ whose endpoints are $v_i$ and $v_3$. By the definition of $\varphi_i$, 
all edges in $P$ are in $\varphi_i(H)$ except 
for $sv_i$ and $v_3v_4$ which are replaced 
by $v_4v_i$ and $sv_3$, respectively. 
Let $P'$ be the union of $sv_3$ followed by $P_1'$, then $v_iv_4$ and finally $P_2$. The endpoints of $P'$ are the terminals $s$ and $t$. 
As $P'$ is in $\varphi(H)$ and $P'$ has precisely as many edges as $P$, we conclude that $d_{\varphi(H)}(s,t)\leq d$. In either case, $\varphi_i(H)$ is in $\mathcal{P}_i^d(Z_n)$, as required. 
It remains to prove that $\varphi_i$ is injective. 
Define $\psi_i:\varphi_i(\mathcal{P}_i^d(Y_n))\to \mathcal{P}_i^d(Y_n)$ as follows:
\begin{equation*}
\psi_i(G)=
\begin{cases}
        G, 
         \text{ if } sv_3 \notin G\\
        G-sv_3+v_3v_4,
         \text{ if } sv_3 \in G \text{ and } G-sv_3+v_3v_4 \in \mathcal{P}_i^d(Y_n)\\
        G-sv_3+v_3v_4-\{v_4v_j: v_4v_j\in G\}
        +\{sv_j: v_4v_j\in G\},
         \text{ if otherwise.}
\end{cases}
\end{equation*}
By construction, $\psi_i(\varphi_i(H))=H$ for every
$H\in \mathcal{P}_i^d(Y_n)$. Hence $\varphi_i$ is injective, and the theorem follows.
\end{proof}

\section{Nonexistence of d-UMRTTGs when d is greater than 3}\label{section:contrast}
We will show in this section that the nonexistence results stated in 
Section~\ref{section:background} are inherited when the distance constraint $d$ is greater than $3$. The key results are Lemma~\ref{lemma:inherit} and Lemma~\ref{lemma:distances}, whose proofs are elementary.
\begin{lemma}\label{lemma:inherit}
Let $d$ be any positive integer and let $G$ be any two-terminal graph in $T_{n,m}$ such that the length of each of its paths joining its terminals is at most $d$. Then, $N_i^d(G)=N_i(G)$ 
for each $i\in \{1,2,\ldots,m\}$. In particular, $R_G^d(\rho)=R_G(\rho)$ for every $\rho$ in $[0,1]$. 
\end{lemma}
\begin{proof}
Let $G$ be any two-terminal graph satisfying the conditions of the statement, and let $i$ be any integer in $\{1,2,\ldots,m\}$. It suffices to prove that $N_i^d(G)=N_i(G)$. By definition, 
$N_i^d(G)\leq N_i(G)$. Consider any spanning subgraph $H$ of $G$ composed of $i$ edges such that there exists some path in $H$ joining the terminals of $G$. By assumption, the length of such a path is at most $d$ thus $H$ is also a $d$-pathset of $G$ and $N_i^d(G)\geq N_i(G)$. Therefore, $N_i^d(G)=N_i(G)$. Equation~\eqref{equation:poly} gives that 
$R_G^d(\rho)=R_G(\rho)$ for every $\rho$ in $[0,1]$, and the lemma follows. 
\end{proof}

\begin{lemma}\label{lemma:distances}
Let $G$ be any two-terminal graph in $T_{n,m}$ and let $d$ be any positive integer. Then, $N_i^d(G)=N_i(G)$ for each $i\in \{1,2,\ldots,d\}$.    
\end{lemma}
\begin{proof}
Let $G$ be any two-terminal graph in $T_{n,m}$ and 
let $d$ and $i$ be positive integers such that $i\in \{1,2,\ldots,d\}$. 
By definition, $N_i^d(G)\leq N_i(G)$. Let $H$ be any spanning subgraph of $G$ composed of $i$ edges 
joining its terminals. Clearly, each path in $H$ joining its terminals has at most $|E(H)|=i$ edges thus $N_i^d(G)\geq N_i(G)$. 
Therefore, $N_i^d(G)=N_i(G)$, and the lemma follows.
\end{proof}

Now we can determine the only $d$-LMRTTG near $1$ in nonempty classes $T_{n,m}$ such that $d\geq 4$.
\begin{lemma}\label{lemma:d4}
Let $d$, $n$ and $m$ be integers such that $d\geq 4$, $n\geq 5$, and $m\geq 5$. All the following assertions hold:
\begin{enumerate}[label=(\roman*)]
    \item\label{it1:low} If $m\leq 2n-3$ then $H_{n,m}$ is the only $d$-LMRTTG near $1$ in $T_{n,m}$.
    \item\label{it2:medium} If $2n-3<m\leq 3n-6$ then $A_{n,m-2n+3}$ is the only $d$-LMRTTG near $1$ in $T_{n,m}$.
    \item\label{it3:high} If $3n-5\leq m\leq \binom{n}{2}-2$ and $G$ is a $d$-LMRTTG near $1$ in $T_{n,m}$ then $\hat{G}$ is not almost-regular. 
\end{enumerate}
\end{lemma}
\begin{proof}
Let us prove each of the assertions separately.
\begin{enumerate}[label=(\roman*)]
    \item Let $G$ be any two-terminal graph in $T_{n,m}$ that is not isomorphic to $H_{n,m}$. 
    By Lemma~\ref{lemma:LMRTTG1}\ref{item1}, the only LMRTTG near $1$ in $T_{n,m}$ is $H_{n,m}$. Then, by Remark~\ref{remark:uniqueness}, there exists $\delta>0$ such that for every $\rho$ in $(1-\delta,1)$ it holds that $R_{H_{n,m}}(\rho)>R_G(\rho)$. Observe that the length of each path joining $s$ and $t$ in $H_{n,m}$ is at most $3$. As $d\geq 4$, Lemma~\ref{lemma:inherit} gives that $R_{H_{n,m}}^d(\rho)=R_{H_{n,m}}(\rho)$ for every $\rho$ in $[0,1]$. As $N_i(G)\geq N_i^d(G)$ for each $i\in \{1,2,\ldots,m\}$, equation~\eqref{equation:poly} gives that $R_G(\rho)\geq R_G^d(\rho)$ for every $\rho$ in $[0,1]$. Then, $R_{H_{n,m}}^d(\rho)=R_{H_{n,m}}(\rho)>R_G(\rho)\geq R_G^d(\rho)$ for every $\rho$ in $(1-\delta,1)$. As $G$ is any two-terminal graph in $T_{n,m}-\{H_{n,m}\}$, we obtained that $H_{n,m}$ is the only $d$-LMRTTG near $1$ in $T_{n,m}$, as required.  
    \item  By Lemma~\ref{lemma:LMRTTG1}\ref{item2}, the only LMRTTG near $1$ in $T_{n,m}$ is $A_{n,m-2n+3}$. Additionally, the length of each path joining $s$ and $t$ in $A_{n,m-2n+3}$ is at most $4$. The proof is analogous to~\ref{it1:low}. 
    \item Let $G$ be any two-terminal graph in $T_{n,m}$ and let $i$ be any integer in $\{1,2,3,4\}$. As $d\geq 4$, Lemma~\ref{lemma:distances} gives that $N_i^d(G)=N_i(G)$. As a consequence, $T_{n,m}^4(i)=T_{n,m}(i)$ for each $i\in \{1,2,3,4\}$. During the proof of Theorem~\ref{theorem:Xie}\ref{teo-not-regular}, Xie et al. proved that each $G$ in $T_{n,m}(4)$ is in $T_{n,m}^u$ and $\hat{G}$ is not almost-regular. Consequently, the set $T_{n,m}^4(4)$ consists of two-terminal graphs $G$ in $T_{n,m}^u$ such that $\hat{G}$ is not almost-regular. As $T_{n,m}^4(m)\subseteq T_{n,m}^4(4)$, the result follows from Lemma~\ref{lemma:existenceLMRTTG}. \qedhere
\end{enumerate}
\end{proof}

We are in a position to prove Theorem~\ref{theorem:nonexistenced>3}
\begin{theorem}\label{theorem:nonexistenced>3}
Let $n$ and $d$ be any pair of integers such that $n\geq 11$ and $d\geq 4$. 
No $d$-UMRTTG exists in $T_{n,m}$ when $20\leq m \leq 3n-9$ or when $3n-5\leq m \leq \binom{n}{2}-2$.
\end{theorem}
\begin{proof}
Let $n$ and $d$ be any pair of integers such that $d\geq 4$ and $n\geq 11$. 
First, let us consider the case in which $20 \leq m \leq 3n-9$. 
Let $G_{n,m}$ be the two-terminal graph in $T_{n,m}$ defined in Lemma~\ref{lemma:half}. 
If $m\leq 2n-3$ then, by Lemma~\ref{lemma:d4}\ref{it1:low}, the only $d$-LMRTTG near $1$ in $T_{n,m}$ is $H_{n,m}$.  Lemma~\ref{lemma:half} gives that $R_{G_{n,m}}(1/2)>R_{H_{n,m}}(1/2)$. As both $G_{n,m}$ and $H_{n,m}$ satisfy the conditions of Lemma~\ref{lemma:inherit} we conclude that $R_{G_{n,m}}^d(1/2)>R_{H_{n,m}}^d(1/2)$ thus by Remark~\ref{remark:uniqueness} 
no $d$-UMRTTG exists. Now, if $2n-3<m\leq 3n-9$ then the reasoning is analogous. Lemma~\ref{lemma:d4}\ref{it2:medium} gives that $A_{n,m-2n+3}$ 
is the only $d$-LMRTTG near $1$ in $T_{n,m}$, while Lemma~\ref{lemma:half} asserts that $R_{G_{n,m}}^d(1/2)>R_{A_{n,m-2n+3}}^d(1/2)$ thus by Remark~\ref{remark:uniqueness} no $d$-UMRTTG exists.

Finally, if $3n-5\leq m \leq \binom{n}{2}-2$ then Lemma~\ref{lemma:d4}\ref{it3:high} gives that each $d$-LMRTTG $G$ near $1$ in $T_{n,m}$ satisfies that $\hat{G}$ is not almost-regular. However, Lemma~\ref{lemma:3-LMRTTG-near0-high} gives that $\hat{G}$ an is almost-regular graph. 
As each $d$-UMRTTG must be simultaneously $d$-LMRTTG near $0$ and $1$, no $d$-UMRTTG exists, and the theorem follows.
\end{proof}

Observe that we still do not know whether or not there exists $d$-UMRTTGs in $T_{n,m}$ when $n\geq 11$ and $m\in \{3n-6,3n-7,3n-8\}$. By Lemma~\ref{lemma:d4}\ref{it2:medium}, we know that the only $d$-LMRTTG near $1$ in $T_{n,m}$ when $m\in \{3n-6,3n-7,3n-8\}$ is $A_{n,n-3}$, $A_{n,n-4}$ and $A_{n,n-5}$, respectively. In the following, we will exhibit a two-terminal graph $B_{n,3n-8}$ in $T_{n,3n-8}$ such that $R_{B_{n,3n-8}}^d(1/2)>R_{A_{n,n-5}}^d(1/2)$ when $d\geq 5$ and $n\geq 11$. The existence or nonexistence of $d$-UMRTTG when $m\in \{3n-6,3n-7\}$ is left for future research. 

Let $G$ be any two-terminal graph in $T_{n,m}$, and let $d$ be any positive integer. 

Let $N^d(G)$ be the number of all $d$-pathsets of $G$, i.e., $N^d(G)=\sum_{i=1}^{m}N_i^d(G)$. Define $F^d(G)$ as the number of all spanning subgraphs of $G$ that are not 
$d$-pathsets. As there are precisely $2^m$ spanning subgraphs of $G$, the sum-rule gives that $N^d(G)+F^d(G)=2^m$. Equation~\eqref{equation:poly} yields,
\begin{equation*}
R_G^d(1/2)= \sum_{i=1}^{m}\frac{N_i^d(G)}{2^m}=\frac{N^d(G)}{2^m} = 1-\frac{F^d(G)}{2^m}.   
\end{equation*}
Therefore, if there exists $B_{n,3n-8}$ in $T_{n,3n-8}$ such that $F^d(B_{n,3n-8})<F^d(A_{n,n-5})$ then $R_{B_{n,3n-8}}^d(1/2)>R_{A_{n,n-5}}^d(1/2)$, 
and by Remark~\ref{remark:uniqueness} no $d$-UMRTTG exists in $T_{n,3n-8}$. We now define $B_{n,3n-8}$.
\begin{definition}
For each integer $n$ such that $n\geq 6$ we define $B_{n,3n-8}$ 
as  $A_{n,n-5}-sv_n-tv_n+v_3v_{n-1}+v_4v_5$. 
\end{definition}

The two-terminal graph $B_{n,3n-8}$ is depicted in Figure~\ref{bnr}. 

\begin{figure}
    \centering
\begin{tikzpicture}
\filldraw[black] (4.5,2) circle (2pt); 
\node[above] (s) at (4.5,2) {$v_2=t$}; 

\filldraw[black] (0,0) circle (2pt); 
\node[below] at (0,0) {$v_3$}; 
\filldraw[black] (1,0) circle (2pt); 
\node[below] at (1,0) {$v_4$}; 
\filldraw[black] (2,0) circle (2pt); 
\node[below] at (2,0) {$v_5$}; 

\filldraw[black] (2.6,0) circle (.5pt); 
\filldraw[black] (3,0) circle (.5pt); 
\filldraw[black] (3.4,0) circle (.5pt); 

\filldraw[black] (3.6,1) circle (.5pt); 
\filldraw[black] (3.8,1) circle (.5pt); 
\filldraw[black] (4,1) circle (.5pt); 
\filldraw[black] (3.6,-1) circle (.5pt); 
\filldraw[black] (3.8,-1) circle (.5pt); 
\filldraw[black] (4,-1) circle (.5pt); 

\filldraw[black] (4,0) circle (2pt); 
\node[below] at (4,0) {$v_{r+2}$}; 
\filldraw[black] (5,0) circle (2pt); 
\node[below] at (5,0) {$v_{r+3}$}; 
\filldraw[black] (6,0) circle (2pt); 
\node[below] at (6,0) {$v_{r+4}$}; 
\filldraw[black] (7,0) circle (2pt); 
\node[below] at (7,0) {$v_{r+5}$}; 

\filldraw[black] (7.6,0) circle (.5pt); 
\filldraw[black] (8,0) circle (.5pt); 
\filldraw[black] (8.4,0) circle (.5pt); 

\filldraw[black] (6,1) circle (.5pt); 
\filldraw[black] (6.2,1) circle (.5pt); 
\filldraw[black] (6.4,1) circle (.5pt); 
\filldraw[black] (6,-1) circle (.5pt); 
\filldraw[black] (6.2,-1) circle (.5pt); 
\filldraw[black] (6.4,-1) circle (.5pt); 

\filldraw[black] (9,0) circle (2pt); 
\node[below] at (9,0) {$v_{n-1}$}; 

\filldraw[black] (10,0) circle (2pt); 
\node[below] at (10,0) {$v_{n}$}; 

\filldraw[black] (4.5,-2) circle (2pt); 
\node[below] (s) at (4.5,-2) {$v_1=s$}; 

\draw (4.5,2) -- (0,0);
\draw (4.5,2) -- (1,0);
\draw (4.5,2) -- (2,0);
%\draw (4.5,1) -- (3,0);
\draw (4.5,2) -- (4,0);
\draw (4.5,2) -- (5,0);
\draw (4.5,2) -- (6,0);
\draw (4.5,2) -- (7,0);
\draw (4.5,2) -- (9,0);

\draw (4.5,-2) -- (0,0);
\draw (4.5,-2) -- (1,0);
\draw (4.5,-2) -- (2,0);
%\draw (4.5,1) -- (3,0);
\draw (4.5,-2) -- (4,0);
\draw (4.5,-2) -- (5,0);
\draw (4.5,-2) -- (6,0);
\draw (4.5,-2) -- (7,0);
\draw (4.5,-2) -- (9,0);

\draw (0,0) -- (1,0);
\draw (0,0) to[bend left] (2,0);
\draw (0,0) to[bend left] (4,0);
\draw (0,0) to[bend left] (5,0);
\draw (0,0) to[bend left] (6,0);
\draw (0,0) to[bend left] (7,0);
\draw (0,0) to[bend left] (9,0);

\draw (1,0) to[bend left] (2,0);

\filldraw[black] (1.2,.35) circle (.5pt); 
\filldraw[black] (1.3,.4) circle (.5pt); 
\filldraw[black] (1.4,.45) circle (.5pt); 

\draw (4.5,2) to[out=0,in=90] (9.5,0);
\draw (9.5,0) to[out=-90,in=0] (4.5,-2);
\end{tikzpicture}
    \caption{\centering The graph $B_{n,3n-8}$}
    \label{bnr}
\end{figure}
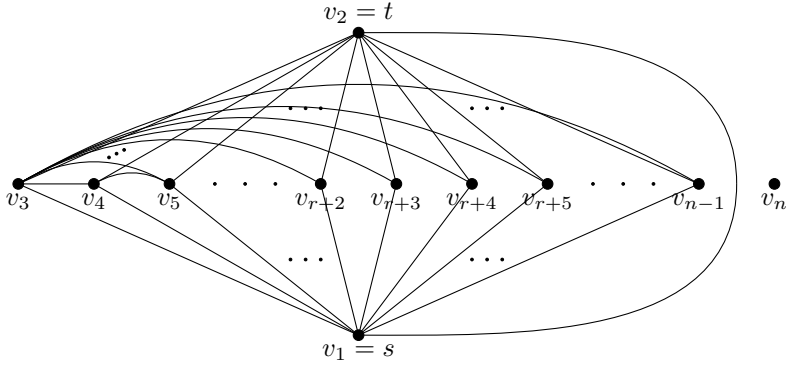

\begin{lemma}\label{lemma:FdeB}
For each integer $n$ such that $n\geq 6$, $F(B_{n,3n-8})=176\times 5^{n-6}-26 \times 4^{n-6}$.    
\end{lemma}

\begin{proof}

We partition the spanning subgraphs $H \subseteq B_{n,3n-8}$ in which $s$ and $t$ are disconnected according to whether the edge $v_4v_5$ belongs to $H$.
As $B_{n,3n-8}-v_4v_5$ is isomorphic to $A_{n-1,n-4}\cup K_1$, the number of spanning subgraphs $H$ of $B_{n,3n-8}$ not containing $v_4v_5$ in which $s$ and $t$ are disconnected equals $F(A_{n-1,n-4})$. By equation~\eqref{eq:F}, 
$F(A_{n-1,n-4})=4\times 5^{n-4}-4^{n-4}$.

We encode each spanning subgraph $H$ by assigning to every vertex 
$v$ in $\{v_3,\dots,v_{n-1}\}$ a state $s(v)$ describing its adjacency to $s$ and $t$ as follows,
\begin{itemize}
    \item $s(v)=B$ if both $sv$ and $vt$ belong to $H$.
    \item $s(v)=N$ if neither $sv$ nor $vt$ belongs to $H$;
    \item $s(v)=S$ if $sv\in H$ and $vt\notin H$;
    \item $s(v)=T$ if $sv\notin H$ and $vt\in H$;

\end{itemize}

We now count all spanning subgraphs $H$ of $B_{n,3n-8}$ including $v_4v_5$ while keeping $s$ and $t$ disconnected. Since $s$ and $t$ are required to remain disconnected in $H$, no vertex is assigned state $B$. We distinguish the three possible states of $v_3$, namely $N$, $S$, and $T$. The cases $S$ and $T$ are symmetric under the automorphism of $B_{n,3n-8}$ induced by exchanging $s$ and $t$, and consequently yield identical contributions. We will show how to construct 
all such spanning subgraphs $H$ meeting the aforementioned conditions 
when $s(v_3)=N$ and $s(v_3)=S$ separately. 
\begin{enumerate}
    \item If $s(v_3)=N$, then $H$ contains no path joining $s$ and $t$ if and only if the following conditions hold:
    \begin{enumerate}
        \item either $s(v)\in\{N,T\}$ for every vertex
        $v\in N_H(v_3)\cap\{v_4,v_5,\ldots,v_{n-1}\}$, or
        $s(v)\in\{N,S\}$ for every such vertex;
        
        \item $s(v)\in\{N,S,T\}$ for every vertex
        $v\in\{v_6,v_7,\ldots,v_{n-1}\}-N_H(v_3)$;
        
        \item $\{s(v_4),s(v_5)\}\neq\{S,T\}$.
    \end{enumerate}

    \item If $s(v_3)=S$, then $H$ contains no path joining $s$ and $t$ if and only if the following conditions hold:
    \begin{enumerate}
        \item $s(v)\in\{N,S\}$ for every vertex
        $v\in N_H(v_3)\cap\{v_4,v_5,\ldots,v_{n-1}\}$;
        
        \item $s(v)\in\{N,S,T\}$ for every vertex
        $v\in\{v_6,v_7,\ldots,v_{n-1}\}-N_H(v_3)$;
        
        \item $\{s(v_4),s(v_5)\}\neq\{S,T\}$.
    \end{enumerate}
\end{enumerate}

We now count the spanning subgraphs in each case separately. Define $
i=|N_H(v_3)\cap\{v_4,v_5\}|$ and
$j=|N_H(v_3)\cap\{v_6,v_7,\ldots,v_{n-1}\}|$. Thus, $i$ counts how many of the vertices in $\{v_4,v_5\}$ are adjacent to $v_3$ in $H$, while $j$ counts the neighbors of $v_3$ among $\{v_6,v_7,\ldots,v_{n-1}\}$. Observe that $i\in\{0,1,2\}$ and $j\in\{0,1,\ldots,n-6\}$.

First, assume that $s(v_3)=N$. If $i=0$, then the pair $(s(v_4),s(v_5))$ can be chosen in $7$ ways subject to the condition $\{s(v_4),s(v_5)\}\neq\{S,T\}$. Observe that, for each such choice, either every neighbor of $v_3$ has state in $\{N,S\}$ or every neighbor of $v_3$ has state in $\{N,T\}$. 
Fix one of these two possibilities. Then every vertex in $N_H(v_3)\cap\{v_6,v_7,\ldots,v_{n-1}\}$ has $2$ possible states, whereas every other vertex has $3$ possible states. As there are precisely $\binom{n-6}{j}$ possible neighbors of $v_3$ when $j$ is a fixed number in $\{0,1,\ldots,n-6\}$, the number of spanning subgraphs $H$ such that $s(v_3)=N$ and $i=0$ equals $\sum_{j=0}^{n-6}\binom{n-6}{j}2^j 3^{n-6-j}=5^{n-6}$. 

Since there are two possible choices for $s(v_3)$, this gives
$2\times 5^{n-6}$. However, the assignments in which every neighbor of $v_3$ has state $N$ are counted twice, once in each choice of $s(v_3)$. 
The number of such assignments equals
$\sum_{j=0}^{n-6}\binom{n-6}{j}3^{n-6-j}=4^{n-6}$. 
Therefore, the number of state assignments equals
$2\times 5^{n-6}-4^{n-6}$. Hence, the number of spanning subgraphs $H$ such that $s(v_3)=N$ and $i=0$ equals
$$
7\times (2\times 5^{n-6}-4^{n-6})
=
14\times 5^{n-6}-7\times 4^{n-6}.
$$

Now, assume that $s(v_3)=N$ and $i\in\{1,2\}$. Then $v_3$ is adjacent to $v_4$, to $v_5$, or to both vertices, yielding $3$ possibilities. Suppose that there exists a vertex in $N_H(v_3)\cap\{v_6,v_7,\ldots,v_{n-1}\}$ with state $S$. Then every neighbor of $v_3$ must have state in $\{N,S\}$. In particular, $s(v_4),s(v_5)\in\{N,S\}$, 
which gives $4$ possibilities for the pair $(s(v_4),s(v_5))$. Hence, this contributes
$$3\times 4\times \sum_{j=1}^{n-6}\binom{n-6}{j}(2^j-1)3^{n-6-j}=12\times (5^{n-6}-4^{n-6})$$
sub-cases. By symmetry, the same contribution arises when there exists a vertex in $N_H(v_3)\cap\{v_6,v_7,\ldots,v_{n-1}\}$ with state $T$. Finally, assume that every vertex in $N_H(v_3)\cap\{v_6,v_7,\ldots,v_{n-1}\}$ has state $N$. Then every vertex in 
$\{v_6,v_7,\ldots,v_n\}-N_H(v_3)$ has state in $\{N,S,T\}$
and the pair $(s(v_4),s(v_5))$ can be chosen in $7$ ways subject to the condition $\{s(v_4),s(v_5)\}\neq\{S,T\}$. Thus, this contributes
$$3\times 7\times\sum_{j=0}^{n-6}\binom{n-6}{j}1^j3^{n-6-j}=21\times 4^{n-6}$$
sub-cases. Therefore, the number of spanning subgraphs $H$ such that $s(v_3)=N$ and $i\in\{1,2\}$ equals
$$2\times (12\times (5^{n-6}-4^{n-6}))+21\times 4^{n-6}=24\times 5^{n-6}-3\times 4^{n-6}.$$

Consequently, the total number of spanning subgraphs $H$ such that $s(v_3)=N$ equals
$$(14\times 5^{n-6}-7\times 4^{n-6})+(24\times 5^{n-6}-3\times 4^{n-6})=38\times 5^{n-6}-10\times 4^{n-6}.$$

Assume now that $s(v_3)=S$. If $i=0$, then $(s(v_4),s(v_5))$ can be chosen in $7$ ways subject to the condition $\{s(v_4),s(v_5)\}\neq\{S,T\}$. Moreover, every vertex adjacent to $v_3$ has state in $\{N,S\}$, whereas every nonadjacent vertex has state in $\{N,S,T\}$. Therefore, the number of spanning subgraphs $H$ equals
$$7\times \sum_{j=0}^{n-6}\binom{n-6}{j}2^j3^{n-6-j}=7\times 5^{n-6}.$$

If $i\in\{1,2\}$, then $v_3$ is adjacent to $v_4$, to $v_5$, or to both vertices, yielding $3$ possibilities. In each case,
$s(v_4),s(v_5)\in\{N,S\}$, 
yielding $4$ possibilities for the pair $(s(v_4),s(v_5))$. Hence, the number of spanning subgraphs $H$ equals
$$3\times 4\times\sum_{j=0}^{n-6}\binom{n-6}{j}2^j3^{n-6-j}=12\times 5^{n-6}.$$

Therefore, the total number of spanning subgraphs $H$ such that $s(v_3)=S$ equals $19\times 5^{n-6}$. 
By symmetry, the number of spanning subgraphs $H$ such that $s(v_3)=T$ also equals $19\times 5^{n-6}$. 

Finally, every spanning subgraph $H$ of $B_{n,3n-8}$ in which $s$ and $t$ are disconnected belongs to exactly one of the following four classes: 
(a) $v_4v_5\notin H$; (b) $v_4v_5\in H$ and $s(v_3)=N$; (c) $v_4v_5\in H$ and $s(v_3)=S$; or (d) $v_4v_5\in H$ and $s(v_3)=T$. 
Therefore,
\begin{align*}
F(B_{n,3n-8}) &= (4\times 5^{n-4}-4^{n-4}) + (38\times 5^{n-6}-10\times 4^{n-6})
+ 2\times (19\times 5^{n-6})\\
&= 176\times 5^{n-6}-26\times 4^{n-6},
\end{align*}
and the lemma follows.
\end{proof}

\begin{lemma}\label{lemma:FddeB}
For each pair of integers $n$ and $d$ such that $n\geq 6$ and $d\geq 5$ it holds that $F^d(B_{n,3n-8})=176\times 5^{n-6}-26 \times 4^{n-6}$. 
\end{lemma}
\begin{proof}
Let $n$ and $d$ be integers such that $n\geq 6$ and $d\geq 5$. 
Observe that the length of each path in $B_{n,3n-8}$ joining its terminals is at most $5$. 
As $d\geq 5$, by Lemma~\ref{lemma:inherit} we obtain that $N_i^d(B_{n,3n-8})=N_i(B_{n,3n-8})$ 
for each $i\in \{1,2,\ldots,m\}$ thus $F_i^d(B_{n,3n-8})=F_i(B_{n,3n-8})$ and $F^d(B_{n,3n-8})=F(B_{n,3n-8})$, as required.
\end{proof}

\begin{proposition}\label{prop:3n-8}
No $d$-UMRTTG exists in $T_{n,3n-8}$ when $d\geq 5$ and $n\geq 11$.  
\end{proposition}
\begin{proof}
Let $n$ and $d$ be integers such that $d\geq 5$ and $n\geq 11$. 
As $d\geq 5$, by Lemma~\ref{lemma:d4}\ref{it2:medium}, the only 
$d$-LMRTTG near $1$ in $T_{n,3n-8}$ is
$A_{n,n-5}$. By Remark~\ref{remark:uniqueness}, it suffices to prove that $R_{B_{n,3n-8}}^d(1/2)>R_{A_{n,n-5}}(1/2)$, or equivalently, that 
$F^d(B_{n,3n-8})<F^d(A_{n,n-5})$. On the one hand, Lemma~\ref{lemma:FddeB} gives that $F^d(B_{n,3n-8})=176\times 5^{n-6}-26 \times 4^{n-6}$. 
On the other hand, as the length of each of the paths joining $s$ and $t$ in $A_{n,n-5}$ is at most $4$, 
Lemma~\ref{lemma:inherit} gives that $N_i^d(A_{n,n-5})=N_i(A_{n,n-5})$ for each $i\in \{1,2,\ldots,m\}$ hence 
$F^d(A_{n,n-5})=F(A_{n,n-5})$. By equation~\eqref{eq:F}, 
$F(A_{n,n-5})=3^{2}\times (4\times 5^{n-5}-4^{n-5})$ thus 
$F^d(A_{n,n-5})=3^{2}\times (4\times 5^{n-5}-4^{n-5})=36\times 5^{n-5}-9\times 4^{n-5} = 180\times 5^{n-6}-36\times 4^{n-6}$. 

Therefore, it suffices to prove that $176\times 5^{n-6}-26 \times 4^{n-6}<180\times 5^{n-6}-36\times 4^{n-6}$, 
or equivalently, that $\left(\frac{5}{4}\right)^{n-6}>\frac{5}{2}$, which holds since $n\geq 11$. The proposition follows.
\end{proof}

\section{Concluding remarks}\label{section:conclusions}
The goal of this work is to determine, as a function of the distance constraint $d$, whether $d$-UMRTTGs exist in each nonempty set $T_{n,m}$. Corollary~\ref{cor:d1} and Theorem~\ref{theorem:2UMRTTG} establish that $1$-UMRTTGs and $2$-UMRTTGs always exist. Theorem~\ref{theorem:almostcomplete} asserts that $Z_n$ is the only $d$-UMRTTG in $T_{n,\binom{n}{2}-1}$ for $d\geq 3$. Since $T_{n,\binom{n}{2}}$ consists only of the complete graph $K_n$ with two terminals, $K_n$ is the unique $d$-UMRTTG for every positive integer $d$.

Theorem~\ref{theorem:nonexistenced>3} establishes the nonexistence of $d$-UMRTTGs for every $d\geq 4$ and every $n\geq 11$ whenever $20\leq m\leq 3n-9$ or $3n-5\leq m\leq \binom{n}{2}-2$. Finally, Proposition~\ref{prop:3n-8} establishes the nonexistence of $d$-UMRTTGs in $T_{n,3n-8}$ for $d\geq 5$ and $n\geq 11$.

Table~\ref{table:summary} summarizes known results on the existence and nonexistence of $d$-UMRTTGs. The existence of $3$-UMRTTGs remains open for all admissible values of $m$, except when $m\in\{\binom{n}{2}-1,\binom{n}{2}\}$, for which $3$-UMRTTGs exist. In addition, the cases with $d=4$ and $m=3n-8$, as well as those with $d\geq 3$ and $m\in\{3n-7,3n-6\}$, are still unresolved.

\begin{table}[ht]
\centering
\renewcommand{\arraystretch}{1.2}
\setlength{\tabcolsep}{2pt}
\scriptsize
\begin{tabular}{|c|c|c|c|c|c|c|c|}
\hline
$d$ 
& $m\leq 2n-3$
& $2n-3<m\leq 3n-9$
& $m=3n-8$
& $m \in \{3n-7,3n-6\}$
& $3n-5\leq m\leq \binom{n}{2}-2$
& $m\in \{\binom{n}{2}-1,\binom{n}{2}\}$
\\
\hline
$1$
& Existence
& Existence
& Existence
& Existence
& Existence
& Existence
\\
\hline
$2$
& Existence
& Existence
& Existence
& Existence
& Existence
& Existence
\\
\hline
$3$
& Open
& Open
& Open
& Open
& Open
& Existence
\\
\hline
$4$
& Nonexistence
& Nonexistence
& Open
& Open
& Nonexistence
& Existence
\\
\hline
$d\geq 5$
& Nonexistence
& Nonexistence
& Nonexistence
& Open
& Nonexistence
& Existence
\\
\hline
\end{tabular}
\caption{Existence or nonexistence of $d$-UMRTTGs in $T_{n,m}$ when $n\geq 11$ and $m\geq 20$.}
\label{table:summary}
\end{table}

\section*{Acknowledgment}
This work is partially supported by City University of New York project entitled \emph{On the problem of characterizing graphs with maximum number of spanning trees} with grant number 66165-00. The author wishes to thank Professor Louis Petingi for his insightful comments which improved the presentation of this manuscript, as well as Professors Martín Safe and Guillermo Durán for their academic support and guidance.

\end{document}